\documentclass[reqno,11pt]{amsart}

\usepackage{youngtab}
\usepackage{hyperref,tikz,amssymb,enumerate}
\usepackage{xparse} 
\usepackage{eucal}
\usepackage{skak} 

\definecolor{color1}{HTML}{E87722} 
\definecolor{color2}{HTML}{009CDE} 
\DeclareMathAlphabet{\mathbfit}{OML}{cmm}{b}{it}  
\newcommand{\cbf}{\textbf}
\newcommand{\cbfit}[1]{\textcolor{color1}{\mathbfit{#1}}}

{\theoremstyle{plain}
\newtheorem{theorem}{Theorem}[section]
\newtheorem*{theorem*}{Theorem}
\newtheorem{proposition}[theorem]{Proposition}
\newtheorem{corollary}[theorem]{Corollary}
}
{\theoremstyle{definition}
\newtheorem{definition}[theorem]{Definition}
\newtheorem*{definition*}{Definition}
\newtheorem{remark}[theorem]{Remark}
\newtheorem{example}[theorem]{Example}
}

\numberwithin{equation}{section}
\numberwithin{figure}{section}

\def\U{\mathsf{U}}
\def\D{\mathsf{D}}
\def\Dyck{\mathsf{Dyck}}
\def\Part{\mathsf{Part}}
\def\nmi{\mathsf{nmi}}

\newcommand{\syt}[1]{\mathrm{SYT}(#1)}
\newcommand{\oeis}[1]{\cite[\oeislink{#1}]{OEIS}}
\newcommand{\oeislink}[1]{\href{https://oeis.org/#1}{#1}}

\usetikzlibrary{shapes,arrows}

\tikzstyle{pathdefault}=[draw, solid, line width=0.8, color=black]
\tikzstyle{nodedefault}=[circle, inner sep=1, fill=black]
\tikzstyle{pathcolorlight}=[draw, line width=1, dotted, color=lightgray]
\tikzstyle{pathmarked}=[draw, solid, line width=1.2, color=color1]
\tikzstyle{nodemarked}=[circle, inner sep=1.5, fill=color1]

\tikzstyle{line} = [draw, -latex']
    
\newcounter{id}
\newcommand{\drawlinedotswithstyle}[4]{
 \def\x{{#3}};
 \def\y{{#4}};
 \tikzstyle{thispathstyle}=[#1]
 \tikzstyle{thisnodestyle}=[#2]
 \setcounter{id}{-1}; 
 \foreach \j in {#3}{\stepcounter{id}} 
 \foreach \i in {1,...,\the\value{id}}{  
  \path[thispathstyle] (\x[\i],\y[\i]) --(\x[\i-1],\y[\i-1]); 
 }
 \foreach \i in {1,...,\the\value{id}}{  
  \node[thisnodestyle] at (\x[\i],\y[\i]) {}; 
 }
 \node[thisnodestyle] at (\x[0],\y[0]) {}; 
}
\DeclareDocumentCommand{\drawlinedots}{ O{pathdefault} O{nodedefault} m m}{\drawlinedotswithstyle{#1}{#2}{#3}{#4}} 

\title[From Dyck paths to standard Young tableaux]{From Dyck paths to standard Young tableaux}

\author[Gil]{Juan B. Gil}
\address{Penn State Altoona, 3000 Ivyside Park, Altoona, PA 16601, USA} 
\email{jgil@psu.edu}

\author[McNamara]{Peter R.~W. McNamara}
\address{Department of Mathematics, Bucknell University, Lewisburg, PA 17837, USA}
\email{peter.mcnamara@bucknell.edu}

\author[Tirrell]{Jordan O. Tirrell}
\address{Department of Mathematics and Statistics, Mount Holyoke College, South Hadley, MA 01075, USA}
\curraddr{Department of Mathematics and Computer Science, Washington College, Chestertown, MD 21620, USA}
\email{jtirrell2@washcoll.edu}

\author[Weiner]{Michael D. Weiner}
\address{Penn State Altoona, 3000 Ivyside Park, Altoona, PA 16601, USA} 
\email{mdw8@psu.edu}

\thanks{Peter McNamara was partially supported by grant \#245597 from the Simons Foundation.}


\begin{document}

\begin{abstract}
We present nine bijections between classes of Dyck paths and classes of standard Young tableaux (SYT). In particular, we consider SYT of flag and rectangular shapes, we give Dyck path descriptions for certain SYT of height at most 3, and we introduce a special class of labeled Dyck paths of semilength $n$ that is shown to be in bijection with the set of all SYT with $n$ boxes. In addition, we present bijections from certain classes of Motzkin paths to SYT. As a natural framework for some of our bijections, we introduce a class of set partitions which in some sense is dual to the known class of noncrossing partitions. 
\end{abstract}

\maketitle

\section{Introduction}

Dyck paths and standard Young tableaux (SYT) are two of the most central sets in combinatorics. Dyck paths of semilength $n$ are perhaps the best-known family counted by the Catalan number $C_n$, while SYT, beyond their beautiful definition, are one of the building blocks for the rich combinatorial landscape of symmetric functions.  

Despite their very different definitions, there are interesting connections between Dyck paths and SYT, led by the elegant bijection between Dyck paths of semilength $n$ and SYT of shape $(n,n)$: for $1 \leq j \leq 2n$, if the $j$th step in the Dyck paths is an up-step (resp.\ down-step), then put the entry $j$ in the first (resp.\ second) row of the SYT. We propose that this bijection is just the tip of the iceberg by establishing a large number of other bijections between subsets of Dyck paths and subsets of SYT, as we elaborate below.    

First, though, let us mention some results of this type from the literature. Pechenik \cite{Pec14} generalized the above bijection to the sets of small Schr\"oder paths (steps of (2,0) allowed) and increasing tableaux (as defined in Remark~\ref{rem:pechenik} below). In a different direction, Regev \cite{Reg81} showed that the number of Motkzin paths (steps of (1,0) allowed) of length $n$ equals the number of SYT with $n$ boxes and at most three rows, a result that was later proved bijectively by Eu \cite{Eu10}. This result was generalized to SYT with at most $2d+1$ rows for any $d\geq1$ by Eu et al.~\cite{E+13}. In \cite{Gud10}, Gudmundsson showed bijectively that for $d = k + p$, the class of Dyck paths of semilength $n$ that begin with at least $k$ successive up-steps, end with at least $p$ successive down-steps, and touch the $x$-axis at least once somewhere between the endpoints, is equinumerous with the class of SYT of shape $(n, n-d)$. Most recently, Garsia and Xin \cite{GX18+} gave a bijection between rational Dyck paths and a particular class of rectangular standard Young tableaux.

Expanding on these results, in this paper we present 10 bijections from classes of Dyck and Motzkin paths to classes of SYT. In particular, we look at SYT of hook, flag, and rectangular shape, and we introduce an interesting class of labeled Dyck paths of semilength $n$ that is shown to be equinumerous with the set of SYT with $n$ boxes. Some of the bijections discussed here are basic, and some are minor variations of known bijections. Nonetheless, we include them all to illustrate our approach and to provide a broader picture of how these combinatorial families interact.

Our first group of bijections share a common step, which is a bijective map $\varphi$ from Dyck paths to a class of set partitions that we call \emph{nomincreasing partitions}. We say that a set partition of $[n]=\{1,\dots,n\}$ is nomincreasing if, when written in standard form, the elements which are not minimal in their block form an increasing sequence. For example, the partition $1237|48|5|69$ is nomincreasing because $23789$ is increasing. On the other hand, the partition $1239|48|5|67$ is not nomincreasing since $23987$ is not increasing. The definition of $\varphi$ together with some properties of nomincreasing partitions will appear at the beginning of Section~\ref{sec:flag_shape}, followed by our first three bijections: 

\begin{enumerate}[\quad(1)]
\itemsep3pt
\item As a basic example of $\varphi$ in action, we give a bijection (Prop.~\ref{prop:hook_peaks}) from Dyck paths of semilength $n$ with $k$ peaks and $k$ returns to SYT of hook shape $(k, 1^{n-k})$. Clearly, both of these sets have cardinality $\binom{n-1}{k-1}$.  
\item\label{ite:flag} Much more interestingly, we give a bijection (Prop.~\ref{prop:flags_peaks}) from Dyck paths of semilength $n$ with $k$ peaks and no singletons to SYT of \emph{flag} shape $(k,k,1^{n-2k})$. A \emph{singleton} in a Dyck path is an ascent of length 1. 
\item Using a result from \cite{BGMW16}, the SYT of flag shape $(k,k,1^{n-k})$ are equinumerous  with Dyck paths of semilength $2n$ with $k$ peaks and all ascents of even length such that an ascent of length $2j$ is followed immediately by a descent of length at least $j$.  This result (Prop.~\ref{prop:flags_peaks2}) is proved using our bijection from~\eqref{ite:flag} above.
\end{enumerate}

In Section~\ref{sec:height3}, we consider tableaux with at most three rows and present the bijections numbered \eqref{ite:three_rows}--\eqref{ite:schroder} below. The first two bijections make use of the map $\varphi$, while the last two use modified versions of the classical bijection from Dyck paths to SYT of shape $(n,n)$.

\begin{enumerate}[\quad(1)]
\itemsep3pt
\setcounter{enumi}{3}
\item\label{ite:three_rows} We give a new bijective proof (Prop.~\ref{prop:height3}) that the number of Dyck paths of semilength $n$ that avoid three consecutive up-steps equals the number of SYT with $n$ boxes and at most $3$ rows. 
In addition, this bijection maps Dyck paths with $s$ singletons to SYT with $s$ columns of odd length. 
\item As a special case (Rem.\ \ref{rem:nnn}) of the previous bijection, we get that SYT of shape $(n,n,n)$ correspond to Dyck paths of semilength $3n$ that avoid three consecutive up-steps, have exactly $n$ singletons, end with $\U^2\D^\ell$ for some $\ell\ge 2$, and such that every subpath starting at the origin has at least as many 1-ascents as 2-ascents.  Here and elsewhere, $\U$ (resp.\ $\D$) denotes an up-step (resp.\ down-step).
\item We already mentioned the result of Gudmundsson involving SYT of shape $(n, n-d)$.  We show (Prop.~\ref{prop:n_n-d}) that 
for $0\le d \le n$, SYT of shape $(n, n-d)$ are in bijection with Dyck paths of semilength $n+1$ having exactly $d+1$ returns.
\item\label{ite:schroder}
The number of tableaux of shape $(n,n)$ with label set $\{1,\ldots,2n-k\}$ such that the rows are strictly increasing and the columns are weakly (resp.\ strictly) increasing are known to be enumerated by the large Schr\"oder numbers \oeis{A006318} (resp.\ small Schr\"oder numbers \oeis{A001003}).  We show (Prop.~\ref{prop:Schr_marked}) that these are in bijection with the number of Dyck paths of semilength $n$ with $k$ marked peaks (resp.\ valleys).
\end{enumerate}

In Section~\ref{sec:nxpt}, we present a more elaborate variation of Dyck paths.  In \cite{AsMa08}, Asinowski and Mansour consider Dyck paths whose $k$-ascents are themselves ``colored'' by Dyck paths of length $2k$, for all $k$.  We consider labels on the ascents of a similar flavor in that we color the ascents with \emph{connected matchings}, an example of which is shown in Fig.~\ref{fig:cm-labeled_dyck}. We call such Dyck paths \emph{cm-labeled Dyck paths}.

\begin{enumerate}[\quad(1)]
\itemsep3pt
\setcounter{enumi}{7}
\item The number of cm-labeled Dyck paths of semilength $n$ with $s$ singletons and $k$-noncrossing labels equals the number of SYT with $n$ boxes, $s$ columns of odd length, and at most $2k-1$ rows.  This bijection relies heavily on a bijection of Burrill et al.~\cite{BCFMM16}.  As a corollary (Cor.~\ref{cor:dyck_syt}), we get that the number of cm-labeled Dyck paths of semilength $n$ equals the number of SYT with $n$ boxes.
\end{enumerate}

Generalizing the above class of SYT of shape $(n,n,n)$, in Section~\ref{sec:rectangular} we consider SYT of shape $(n^d)$ and use a result of Wettstein \cite{Wett17} to connect them with Dyck paths of semilength $d\cdot n$ whose ascents are labeled by certain balanced bracket expressions over an alphabet with $d$ letters.

\begin{enumerate}[\quad(1)]
\itemsep3pt
\setcounter{enumi}{8}
\item The set of SYT of shape $(n^d)$ is in bijection with the set of Dyck paths of semilength $d\cdot n$ created from strings of the form $\D$ and $\U^{d\cdot j}\D$ for $j=1,\ldots, n$, and such that each $d\!\cdot\!j$-ascent may be labeled in $p_j$ different ways, where $(p_n)$ is the sequence of $d$-dimensional prime Catalan numbers.
\end{enumerate}

Given Corollary~\ref{cor:dyck_syt}, it is natural to ask for other sets of paths that are in bijection with the \emph{full set} of SYT with $n$ boxes. In Section~\ref{sec:Motzkin}, we present three classes of Motzkin paths with such a bijection.  See Proposition~\ref{prop:MotzkinPaths} for the definitions. In each case, the number of flat steps $s$ in the Motzkin path equals the number of columns of odd length in the tableau.

\begin{enumerate}[\quad(1)]
\itemsep3pt
\setcounter{enumi}{9}
\item The following classes of Motzkin paths with $n$ steps are in bijection with the set of SYT with $n$ boxes.
  \begin{enumerate}
  \item Height-labeled Motzkin paths, whose bijection to SYT is somewhat well known.
  \item Full rook Motzkin paths. In the case of Dyck paths ($s=0$), see \cite{Cal09,Fra78,FV79,Kratt06}.
  \item Yamanouchi-colored Motzkin paths, for which a different bijection to the one we use is given by Eu et al.~\cite{E+13}.
  \end{enumerate}
\end{enumerate}

We conclude in Section~\ref{sec:further_remarks} with several remarks. We put some of our bijections in a more general framework of maps between Dyck paths and restricted set partitions, thereby explaining the assertion in the abstract about the duality between nomincreasing and noncrossing partitions. After a brief discussion of the noncrossing partition transform, we obtain an elegant expression for the generating function for SYT of height at most $2k-1$ in terms of the generating function for $k$-noncrossing perfect matchings.  

\section{Dyck paths to SYT of hook and flag shape} 
\label{sec:flag_shape}

In this section, we will discuss several bijections between Dyck paths with certain restrictions and SYT of special shapes.  We start by defining a bijective map from the set of Dyck paths of semilength $n$ to the set of nomincreasing partitions of $[n]$ that serves as a unifying feature of several of these bijections. Specifically, we let
\begin{equation}\label{map_phi}
 \varphi: \Dyck(n) \to \Part_\nmi(n)
\end{equation}
be defined as follows:
\begin{itemize}
\item From left to right, number the down-steps of the Dyck path with $[n]$ in increasing order. 
\item At each peak $\U\D$, label the up-step with the number already assigned to its paired down-step.
\item Going through the ascents from left to right, label the remaining up-steps from top to bottom on each ascent in a greedy fashion.
\item The resulting labeling gives a nomincreasing partition of $[n]$ whose blocks are the labels on the ascents.
\end{itemize}
For example, the path in Fig.~\ref{fig:peak-greedy} gives the partition $1237|48|5|69$.  As in Fig.~\ref{fig:peak-greedy}, we will represent such a partition by a tableau-like array where the column entries are increasing from top to bottom and give the blocks of the partition while the top row is also increasing and contains the smallest entry from each block; when such an array comes from a nomincreasing partition, we call it a \emph{modified tableau}.  

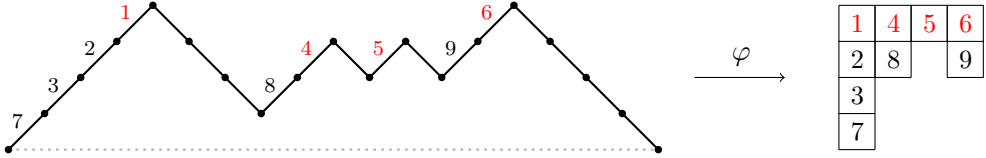
\begin{figure}[ht]
\begin{tikzpicture}[scale=0.4]
\begin{scope}
\scriptsize
	\draw[color1] (3.7,3.8) node[left=1pt] {\cbf{1}};
	\draw[color1] (8.7,2.8) node[left=1pt] {\cbf{4}};
	\draw[color1] (10.7,2.8) node[left=1pt] {\cbf{5}};
	\draw[color1] (13.7,3.8) node[left=1pt] {\cbf{6}};
	\draw (2.7,2.8) node[left=1pt] {2};
	\draw (1.7,1.8) node[left=1pt] {3};
	\draw (0.7,0.8) node[left=1pt] {7};
	\draw (7.7,1.8) node[left=1pt] {8};
	\draw (12.7,2.8) node[left=1pt] {9};
\draw[pathcolorlight] (0,0) -- (18,0);
\drawlinedots{0,1,2,3,4,5,6,7,8,9,10,11,12,13,14,15,16,17,18}%
{0,1,2,3,4,3,2,1,2,3,2,3,2,3,4,3,2,1,0}
\end{scope}

\draw[->] (18.8,2) -- (21.3,2);
\draw (20.2,2) node[above=0.5pt] {$\varphi$};

\begin{scope}[xshift=225mm]
\small
\draw[step=1] (0,0) grid (1,4);
\draw[step=1] (1,2) grid (2,4);
\draw (2,3) rectangle (3,4);
\draw[step=1] (3,2) grid (4,4);
\draw[color1] (0,3.5) node[right=0.5pt] {\cbf{1}};  
\draw (0,2.5) node[right=0.5pt] {2};  
\draw (0,1.5) node[right=0.5pt] {3};  
\draw (0,0.5) node[right=0.5pt] {7};  
\draw[color1] (1,3.5) node[right=0.5pt] {\cbf{4}};  
\draw (1,2.5) node[right=0.5pt] {8};  
\draw[color1] (2,3.5) node[right=0.5pt] {\cbf{5}};  
\draw[color1] (3,3.5) node[right=0.5pt] {\cbf{6}};  
\draw (3,2.5) node[right=0.5pt] {9};  
\end{scope}
\end{tikzpicture}
\caption{Example of the map $\varphi$, where the columns of the array on the right are the blocks of the resulting nomincreasing partition.}
\label{fig:peak-greedy}
\end{figure}

Note that the difference of the smallest entries in two consecutive blocks is the number of down-steps between the corresponding ascents on the path. 

Finally, given a partition $B_1|\cdots|B_\ell$ of $[n]$ with $1=b_1<\cdots<b_\ell$ where $b_i=\min(B_i)$, the reverse map is exactly
\begin{equation}\label{eq:Dyck_from_Part}
\U^{\#B_1}\D^{b_2-b_1}\U^{\#B_2}\D^{b_3-b_2}\cdots\U^{\#B_{\ell-1}}\D^{b_\ell-b_{\ell-1}}\U^{\#B_\ell}\D^{n+1-b_\ell}.
\end{equation} We will modify $\varphi$ to obtain bijections between certain Dyck paths and SYT.

\subsection*{Hook shapes}

We begin with a ``warm up'' example for the use of $\varphi$. An SYT is said to be of \emph{hook shape} if its shape is $(k,1^\ell)$ for some $k$ and $\ell$, where $1^\ell$ denotes a sequence consisting of $\ell$ copies of 1. A Dyck path of semilength $n$ with $k$ peaks and $k$ returns is a Dyck path of the form $\U^{j_1}\D^{j_1} \cdots \U^{j_k}\D^{j_k}$ with $j_1+\dots+j_k=n$. There are $\binom{n-1}{k-1}$ such paths, which is also the number of SYT of shape $(k,1^{n-k})$. To demonstrate the use of $\varphi$, we give a bijective proof of this equinumerosity.

\begin{proposition}\label{prop:hook_peaks}
For $1 \leq k \leq n$, Dyck paths of semilength $n$ with $k$ peaks and $k$ returns are in bijection with SYT of shape $(k,1^{n-k})$.
\end{proposition}

\begin{proof}
Given a Dyck path of the form $\U^{j_1}\D^{j_1} \cdots \U^{j_k}\D^{j_k}$ with $j_1+\dots+j_k=n$, we apply the map $\varphi$ to get the partition 
\[ 1,\dots, j_1\mid j_1+1,\dots, j_1+j_2\mid \cdots\mid n-j_k+1,\dots,n, \] 
which can be represented as a modified tableau. We then obtain an SYT of hook shape by pushing all the boxes below the first row into the first column.  For example, 

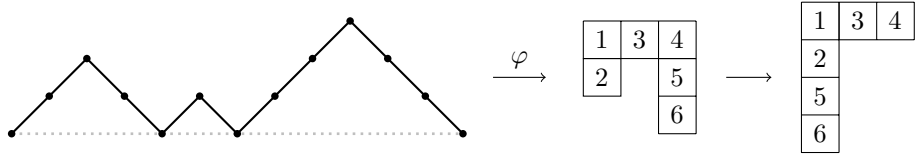
\begin{center}
\begin{tikzpicture}[scale=0.4]
\begin{scope}
\draw[pathcolorlight] (0,0) -- (12,0);
\drawlinedots{0,1,2,3,4,5,6,7,8,9,10,11,12}{0,1,2,1,0,1,0,1,2,3,2,1,0}
\draw[->] (12.8,1.5) -- (14.2,1.5);
\draw (13.5,1.4) node[above=1pt] {\small $\varphi$};
\end{scope}

\begin{scope}[xshift=152mm]
\small
\draw[step=1] (0,2) grid (3,3);
\draw (0,1) rectangle (1,2);
\draw[step=1] (2,0) grid (3,2);
\draw (0,2.5) node[right=0.5pt] {1};  
\draw (0,1.5) node[right=0.5pt] {2};  
\draw (1,2.5) node[right=0.5pt] {3};  
\draw (2,2.5) node[right=0.5pt] {4};  
\draw (2,1.5) node[right=0.5pt] {5};  
\draw (2,0.5) node[right=0.5pt] {6};  
\draw[->] (3.8,1.5) -- (5,1.5);
\end{scope}

\begin{scope}[xshift=210mm, yshift=-0.5cm]
\small
\draw[step=1] (0,0) grid (1,4);
\draw[step=1] (1,3) grid (3,4);
\draw (0,3.5) node[right=0.5pt] {1};  
\draw (0,2.5) node[right=0.5pt] {2};  
\draw (0,1.5) node[right=0.5pt] {5};  
\draw (0,0.5) node[right=0.5pt] {6};  
\draw (1,3.5) node[right=0.5pt] {3};  
\draw (2,3.5) node[right=0.5pt] {4};  
\draw (4,2) node {.};
\end{scope}
\end{tikzpicture}
\end{center}

For the inverse, let $a_1,a_2,\dots,a_k$ be the entries of the first row of a given SYT of shape $(k,1^{n-k})$. Move the boxes that appear below the first row to the unique place such that the modified tableau $T$ has columns with increasing consecutive entries. The length of column $i$ in $T$ is then the length of the $i$th ascent (from left to right) on the Dyck path, which uniquely determines a Dyck path with $k$ peaks and $k$ returns.
\end{proof}

\begin{corollary}
The number of Dyck paths of semilength $n$ with as many peaks as returns equals the number of SYT of hook shape with $n$ boxes.
\end{corollary}

\subsection*{Flag shapes}

We next consider results related to SYT of shape $(k,k,1^{n-2k})$, which we will refer to as SYT of \emph{flag shape}.  Using the hook-length formula, one can check that the number of such tableaux is 
\begin{equation}\label{equ:hooklength}
\frac{1}{n+1}\binom{n+1}{k}\binom{n-k-1}{k-1}.
\end{equation}

For a fixed integer $k\ge 1$, Stanley~\cite{Sta96} gave a bijection from dissections of an $(n-k+2)$-gon using exactly $k-1$ diagonals to SYT of shape $(k,k,1^{n-2k})$. We will give a bijection that extends this to the Dyck path setting. Analogous to the way that Narayana numbers refine Catalan numbers by considering the number of peaks, we get the following result.

\begin{proposition}\label{prop:flags_peaks}
For $1 \leq k \leq \lfloor \frac{n}{2} \rfloor$, Dyck paths of semilength $n$ with $k$ peaks and no singletons are in bijection with SYT of shape $(k,k,1^{n-2k})$.
\end{proposition}

\begin{proof}
We present the bijection using the illustrative example:

\medskip
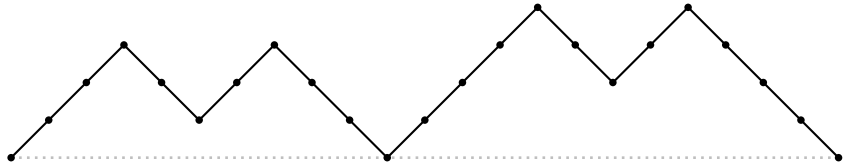
\begin{center}
\begin{tikzpicture}[scale=0.45]
\draw[pathcolorlight] (0,0) -- (22,0);
\drawlinedots{0,1,2,3,4,5,6,7,8,9,10,11,12,13,14,15,16,17,18,19,20,21,22}%
{0,1,2,3,2,1,2,3,2,1,0,1,2,3,4,3,2,3,4,3,2,1,0}
\draw (22.5, 2) node {.};
\end{tikzpicture}
\end{center}
We apply the map $\varphi$ to a Dyck path of semilength $n$ with no singletons and represent the resulting partition of $[n]$ as a modified tableau:

\begin{center}
\begin{tikzpicture}[scale=0.45]
\small
\draw[step=1] (0,2) grid (4,4);
\draw (0,1) rectangle (1,2);
\draw[step=1] (2,0) grid (3,2);
\draw (0,3.5) node[right=1pt] {1};  
\draw (0,2.5) node[right=1pt] {2};  
\draw (0,1.5) node[right=1pt] {4};  
\draw (1,3.5) node[right=1pt] {3};  
\draw (1,2.5) node[right=1pt] {5};  
\draw (2,3.5) node[right=1pt] {6};  
\draw (2,2.5) node[right=1pt] {7};  
\draw (2,1.5) node[right=1pt] {9};  
\draw (2.5,0.5) node {10};  
\draw (3,3.5) node[right=1pt] {8};  
\draw (3.5,2.5) node {11};  
\draw (4.5, 2) node {.};
\end{tikzpicture}
\end{center}
The SYT of flag shape is then produced by pushing all the boxes below the second row into the first column:

\begin{center}
\begin{tikzpicture}[scale=0.45]
\small
\draw[step=1] (0,2) grid (4,4);
\draw[step=1] (0,-1) grid (1,2);
\draw (0,3.5) node[right=1pt] {1};  
\draw (0,2.5) node[right=1pt] {2};  
\draw (0,1.5) node[right=1pt] {4};  
\draw (0,0.5) node[right=1pt] {9};  
\draw (0.5,-0.5) node {10};  
\draw (1,3.5) node[right=1pt] {3};  
\draw (1,2.5) node[right=1pt] {5};  
\draw (2,3.5) node[right=1pt] {6};  
\draw (2,2.5) node[right=1pt] {7};  
\draw (3,3.5) node[right=1pt] {8};  
\draw (3.5,2.5) node {11};  
\draw (4.5,1.5) node {.};
\end{tikzpicture}
\end{center}

Conversely, given an SYT of shape $(k,k,1^{n-2k})$, let us call the entries of the first row $a_1, a_2, \ldots, a_k$ from left to right, and let us use $b_1, b_2, \ldots, b_k$ for the entries in the second row. We rearrange the boxes below the second row by moving the box containing the number $j$ into the unique column $i$ whereby $b_i < j < b_{i+1}$ (where we let $b_{k+1}=n+1$), thus yielding a modified tableau.  Applying $\varphi^{-1}$ completes the inverse map.
\end{proof}

Summing over $k=1,\dots,\lfloor \frac{n}{2} \rfloor$, we recover two manifestations of the sequence \oeis{A005043} of ``Riordan numbers.''

\begin{corollary}
The number of Dyck paths of semilength $n$ without singleton ascents equals the number of SYT of flag shape with $n$ boxes.
\end{corollary}

\begin{remark}\label{rem:pechenik}
There is a less direct way to construct a bijection that proves Proposition~\ref{prop:flags_peaks} using results already in the literature.  An \emph{increasing tableau} is a semistandard Young tableau whose rows and columns are strictly increasing and the set of entries is an initial segment of the positive integers. In \cite{Pec14}, Pechenik gives a bijection from SYT of shape $(k,k,1^{n-2k})$ to increasing tableaux of shape $(n-k,n-k)$ whose maximum entry is at most $n$.  He also provides a bijection from such increasing tableaux to noncrossing partitions of $n$ into $k$ blocks each of size at least 2. There is a well-known bijection between these noncrossing partitions and Dyck paths of semilength $n$ with $k$ peaks and no singletons, as required.  
\end{remark}

Another connection between Dyck paths and SYT of flag shape begins with a result from \cite{BGMW16}.  A special case of the Dyck paths considered there is the set $\mathfrak{D}_n(1,1)$, which denotes the set of Dyck paths of semilength $2n$ created from strings of the form $\D$ and $\U^{2j}\D^j$ for $j=1,\ldots, n$.  In \cite[Theorem~3.5]{BGMW16}, the number of such Dyck paths with exactly $k$ peaks is shown to be
\begin{equation*}
\frac{1}{k}\binom{n+k}{k-1}\binom{n-1}{k-1} = \frac{1}{n+k+1}\binom{n+k+1}{k}\binom{n-1}{k-1}.
\end{equation*}
This is exactly the number of SYT of shape $(k,k,1^{n-k})$, cf.\ \eqref{equ:hooklength}. Thus we have:

\begin{proposition}\label{prop:flags_peaks2}
For $1 \leq k \leq n$, Dyck paths in $\mathfrak{D}_n(1,1)$ with $k$ peaks 
are in bijection with SYT of shape $(k,k,1^{n-k})$.
\end{proposition}

\begin{example}
For $n=2$, the three elements of $\mathfrak{D}_2(1,1)$ are
\medskip
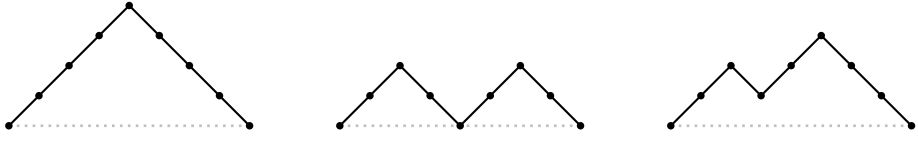
\begin{center}
\begin{tikzpicture}[scale=0.35]
\begin{scope}
\draw[pathcolorlight] (0,0) -- (8,0);
\drawlinedots{0,1,2,3,4,5,6,7,8}{0,1,2,3,4,3,2,1,0}
\end{scope}
\begin{scope}[xshift=110mm]
\draw[pathcolorlight] (0,0) -- (8,0);
\drawlinedots{0,1,2,3,4,5,6,7,8}{0,1,2,1,0,1,2,1,0}
\end{scope}
\begin{scope}[xshift=220mm]
\draw[pathcolorlight] (0,0) -- (8,0);
\drawlinedots{0,1,2,3,4,5,6,7,8}{0,1,2,1,2,3,2,1,0}
\end{scope}
\end{tikzpicture}
\end{center}
\medskip
\noindent
and the three SYT are
{\small
\[
\Yvcentermath1
\young(1,2,3)\qquad\quad \young(13,24)\qquad\quad \young(12,34)\ .
\]
}
\end{example}

\medskip
Here is a bijective proof of Proposition~\ref{prop:flags_peaks2}: starting with an element of $\mathfrak{D}_n(1,1)$ with $k$ peaks, replace each building block $\U^{2j}\D^j$ by $\U^{j+1}\D$ to obtain a Dyck path of semilength $n+k$ with $k$ peaks and no singleton ascents. Then apply the bijection from Proposition~\ref{prop:flags_peaks}. 

\section{Dyck paths to SYT of height at most 3}
\label{sec:height3}

It is known that SYT with $n$ boxes and at most $3$ rows are in one-to-one correspondence with the set of Motzkin paths of length $n$ (see \cite{Reg81} and \cite{Eu10}), enumerated by the sequence \oeis{A001006}. On the other hand, Motzkin paths of length $n$ are in bijection with Dyck paths of semilength $n$ that avoid three consecutive up-steps. In other words, we have the following correspondence that we will prove here bijectively using the map $\varphi$.

\begin{proposition}\label{prop:height3}
The number of Dyck paths of semilength $n$ that avoid three consecutive up-steps equals the number of SYT with $n$ boxes and at most $3$ rows. 
\end{proposition}
\begin{proof}
\newcommand{\column}{\genfrac{[}{]}{0pt}{}}
Let $D$ be a Dyck path of semilength $n$ having $m$ peaks and avoiding three consecutive up-steps. We apply $\varphi$ from \eqref{map_phi} to $D$ and call the columns of the corresponding modified tableau $v_1,\dots,v_m$, where each $v_\ell$ is of the form $[x_\ell]$ or $\column{x_\ell}{y_\ell}$. If the modified tableau is an SYT, we are done. If not, we repeatedly apply the following algorithm until an SYT is obtained:
\begin{itemize}
\item Let $j$ be the index of the leftmost column of length 1 and let $v_k$ be the first column of length 2 to the right of $v_j$. If $j=1$, let $y_0=0$. Empty column $k$ according to the following rules.
\item If $x_k>y_{j-1}$, place $y_k$ in the third row and move $x_k$ to the second row of column $j$ so that $y_j=x_k$.
 \item If $x_k<y_{j-1}$, let $i$ be the largest index such that $y_i<x_k$, or set $i=0$ if no such $y_i$ exists. We then place $y_{i+1}$ in the third row, move $x_k$ to $y_{i+1}$'s previous position, and move $y_k$ to the second row of column $j$ so that $y_j=y_k$.
\item Slide the new element in the third row to the left as much as possible, and fill column $k$ by shifting to the left all columns $v_i$ with $i>k$.
\end{itemize}
Since the elements in the third row all come from the second row, and we are placing them in increasing order, the algorithm is guaranteed to create an SYT of height at most 3.

For example,
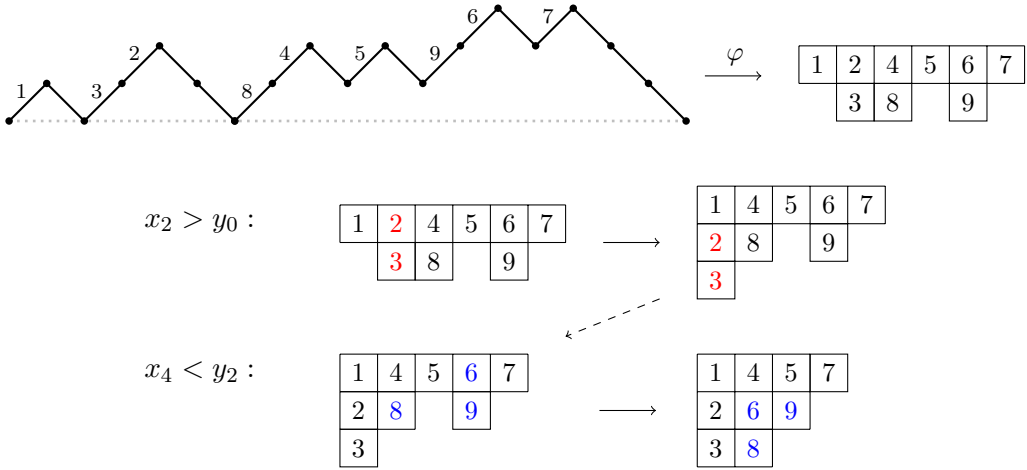
\begin{center}
\begin{tikzpicture}[scale=0.4]
\begin{scope}
\scriptsize
	\draw (0.7,0.8) node[left=0.1pt] {1};
	\draw (3.7,1.8) node[left=0.1pt] {2};
	\draw (2.7,0.8) node[left=0.1pt] {3};
	\draw (7.7,1.8) node[left=0.1pt] {4};
	\draw (6.7,0.8) node[left=0.1pt] {8};
	\draw (9.7,1.8) node[left=0.1pt] {5};
	\draw (12.7,2.8) node[left=0.1pt] {6};
	\draw (11.7,1.8) node[left=0.1pt] {9};
	\draw (14.7,2.8) node[left=0.1pt] {7};
\draw[pathcolorlight] (0,0) -- (18,0);
\drawlinedots{0,1,2,3,4,5,6}{0,1,0,1,2,1,0}
\drawlinedots{6,7,8,9,10,11,12,13,14,15,16,17,18}{0,1,2,1,2,1,2,3,2,3,2,1,0}
\draw[->] (18.5,1.2) -- (20,1.2);
\draw (19.3,1.2) node[above=1pt] {\small $\varphi$};
\end{scope}
\begin{scope}[xshift=210mm]
\small
\draw[step=1] (0,1) grid (6,2);
\draw (1,0) rectangle (2,1);
\draw (2,0) rectangle (3,1);
\draw (4,0) rectangle (5,1);
\draw (0,1.5) node[right=0.5pt] {1};  
\draw (1,0.5) node[right=0.5pt] {3};  
\draw (1,1.5) node[right=0.5pt] {2};  
\draw (2,1.5) node[right=0.5pt] {4};  
\draw (2,0.5) node[right=0.5pt] {8};  
\draw (3,1.5) node[right=0.5pt] {5};  
\draw (4,1.5) node[right=0.5pt] {6};  
\draw (4,0.5) node[right=0.5pt] {9};  
\draw (5,1.5) node[right=0.5pt] {7};  
\end{scope}
\end{tikzpicture}

\vspace{2em}
\begin{tabular}{ll}
$x_2>y_0:$ & \qquad
\begin{tikzpicture}[scale=0.45,baseline=20pt]
\begin{scope}
\small
\draw[step=1] (0,1) grid (6,2);
\draw (1,0) rectangle (2,1);
\draw (2,0) rectangle (3,1);
\draw (4,0) rectangle (5,1);
\draw (0,1.5) node[right=1pt] {1};  
\draw (1,0.5) node[color1,right=1pt] {\cbf{3}};  
\draw (1,1.5) node[color1,right=1pt] {\cbf{2}};  
\draw (2,1.5) node[right=1pt] {4};  
\draw (2,0.5) node[right=1pt] {8};  
\draw (3,1.5) node[right=1pt] {5};  
\draw (4,1.5) node[right=1pt] {6};  
\draw (4,0.5) node[right=1pt] {9};  
\draw (5,1.5) node[right=1pt] {7};  
\draw[->] (7,1) -- (8.5,1);
\end{scope}
\begin{scope}[xshift=95mm,yshift=-5mm]
\small
\draw[step=1] (0,2) grid (5,3);
\draw[step=1] (0,1) grid (2,2);
\draw (0,0) rectangle (1,1);
\draw (3,1) rectangle (4,2);
\draw (0,2.5) node[right=1pt] {1};  
\draw (0,1.5) node[color1,right=1pt] {\cbf{2}};  
\draw (0,0.5) node[color1,right=1pt] {\cbf{3}};  
\draw (1,2.5) node[right=1pt] {4};  
\draw (1,1.5) node[right=1pt] {8};  
\draw (2,2.5) node[right=1pt] {5};  
\draw (3,2.5) node[right=1pt] {6};  
\draw (3,1.5) node[right=1pt] {9};  
\draw (4,2.5) node[right=1pt] {7};  
\end{scope}
\draw[->,dashed] (8.5,-0.5) -- (6,-1.5);
\end{tikzpicture}
\\[4em]
$x_4<y_2:$ & \qquad
\begin{tikzpicture}[scale=0.45,baseline=30pt]
\begin{scope}
\small
\draw[step=1] (0,2) grid (5,3);
\draw[step=1] (0,1) grid (2,2);
\draw (0,0) rectangle (1,1);
\draw (3,1) rectangle (4,2);
\draw (0,2.5) node[right=1pt] {1};  
\draw (0,1.5) node[right=1pt] {2};  
\draw (0,0.5) node[right=1pt] {3};  
\draw (1,2.5) node[right=1pt] {4};  
\draw (1,1.5) node[color2,right=1pt] {\cbf{8}};  
\draw (2,2.5) node[right=1pt] {5}; 
\draw (3,2.5) node[color2,right=1pt] {\cbf{6}};  
\draw (3,1.5) node[color2,right=1pt] {\cbf{9}};  
\draw (4,2.5) node[right=1pt] {7};  
\draw[->] (6.9,1.5) -- (8.4,1.5);
\end{scope}
\begin{scope}[xshift=95mm]
\small
\draw[step=1] (0,1) grid (3,3);
\draw[step=1] (0,0) grid (2,1);
\draw (3,2) rectangle (4,3);
\draw (0,2.5) node[right=1pt] {1};  
\draw (0,1.5) node[right=1pt] {2};  
\draw (0,0.5) node[right=1pt] {3};  
\draw (1,2.5) node[right=1pt] {4};  
\draw (1,1.5) node[color2,right=1pt] {\cbf{6}};  
\draw (1,0.5) node[color2,right=1pt] {\cbf{8}};  
\draw (2,2.5) node[right=1pt] {5};  
\draw (2,1.5) node[color2,right=1pt] {\cbf{9}};  
\draw (3,2.5) node[right=1pt] {7};  
\end{scope}
\end{tikzpicture}
\end{tabular}
\end{center}

\medskip
The above algorithm can be reversed. Let $T$ be an SYT of height at most 3 with $n$ boxes and columns $v_1,v_2,\dots,v_m$. Thus each $v_\ell$ is of the form $[x_\ell]$, $\column{x_\ell}{y_\ell}$, or $\begin{bmatrix}x_\ell\\y_\ell\\z_\ell \end{bmatrix}$. 

If $T$ has height 3, slide the elements of the third row to the right as much as possible subject to the restriction that the columns must have increasing entries.  Then repeatedly apply the algorithm below until there are no more columns of length 3.

\begin{itemize}
\item Let $j$ be the index of the rightmost column of length 3, and let $k$ be the largest index such that $x_{k}<y_j$. Note that, by definition, the entry $z_j$ must be smaller than any existing $y_i$ with $i>j$.
\item Shift the columns $v_{k+1},v_{k+2},\dots$ to the right ($v_i\to v_{i+1}$ for all $i\ge k+1$).
\item If $k=j$ or if $v_{j+1},\dots,v_k$ are all columns of length 1, then insert $\column{y_j}{z_j}$ as the new column $v_{k+1}$, removing $y_j$ and $z_j$ from their previous positions.
\item If $k>j$ and if $v_\ell$ is the rightmost column of length 2 with $j+1\le\ell\le k$, then let $v_{k+1} = \column{y_j}{y_\ell}$, removing $y_j$ and $y_\ell$ from their previous positions, and move $z_j$ to $y_j$'s previous position. Thus the modified columns $v_j$ and $v_\ell$ have lengths 2 and 1, respectively. 
\end{itemize}

Finally, given a tableau with two rows and $m$ columns (standard or modified), we apply $\varphi^{-1}$ to yield a Dyck path of the appropriate type.
\end{proof}

\def\DD{\mathcal{D}_{3\textsc{cat}}}

\begin{remark}\label{rem:nnn}
The above bijection maps Dyck paths with $s$ singletons to SYT with $s$ columns of odd length. Also, it is not hard to see that SYT of shape $(n,n,n)$ correspond to Dyck paths of semilength $3n$ that avoid three consecutive up-steps, have exactly $n$ singletons, end with $\U^2\D^\ell$ for some $\ell\ge 2$, and such that every subpath starting at the origin has at least as many 1-ascents as 2-ascents. We denote this class of special Dyck paths by $\DD(n)$. 

For example, if $n=2$, there are five such SYT:

{\small
\begin{equation*}
\young(12,34,56)\qquad \young(12,35,46)\qquad \young(13,24,56)\qquad \young(13,25,46)\qquad \young(14,25,36)
\end{equation*}}

\medskip\noindent
corresponding to the five paths in $\DD(2)$:

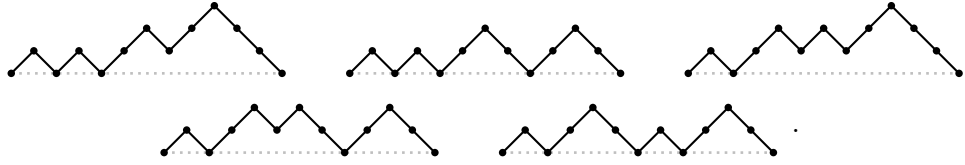
\medskip
\begin{center}
\begin{tikzpicture}[scale=0.28]
\begin{scope}
\draw[pathcolorlight] (0,0) -- (12,0);
\drawlinedots{0,1,2,3,4,5,6,7,8,9,10,11,12}{0,1,0,1,0,1,2,1,2,3,2,1,0}
\end{scope}
\begin{scope}[xshift=140mm]
\draw[pathcolorlight] (0,0) -- (12,0);
\drawlinedots{0,1,2,3,4,5,6,7,8,9,10,11,12}{0,1,0,1,0,1,2,1,0,1,2,1,0}
\end{scope}
\begin{scope}[xshift=280mm]
\draw[pathcolorlight] (0,0) -- (12,0);
\drawlinedots{0,1,2,3,4,5,6,7,8,9,10,11,12}{0,1,0,1,2,1,2,1,2,3,2,1,0}
\end{scope}
\end{tikzpicture}
\end{center}

\medskip
\begin{center}
\begin{tikzpicture}[scale=0.28]
\begin{scope}
\draw[pathcolorlight] (0,0) -- (12,0);
\drawlinedots{0,1,2,3,4,5,6,7,8,9,10,11,12}{0,1,0,1,2,1,2,1,0,1,2,1,0}
\end{scope}
\begin{scope}[xshift=140mm]
\draw[pathcolorlight] (0,0) -- (12,0);
\drawlinedots{0,1,2,3,4,5,6,7,8,9,10,11,12}{0,1,0,1,2,1,0,1,0,1,2,1,0}
\draw (13,1) node {.};
\end{scope}
\end{tikzpicture}
\end{center}
\end{remark}

\begin{proposition} \label{prop:nnn_syt}
The set of SYT of shape $(n,n,n)$ is in bijection with the set of Dyck paths in $\DD(n)$. By the hook-length formula, these sets are enumerated by
\[ \frac{2(3n)!}{n!(n+1)!(n+2)!}, \]
which is the sequence \oeis{A005789} of 3-dimensional Catalan numbers.
\end{proposition}

\subsection*{Tableaux with two rows}

In the remaining part of this section, we modify the classic bijection between Dyck paths and standard Young tableaux of shape $(n, n)$ to describe SYT of shape $(n, n-d)$ and some nonstandard tableaux of shape $(n,n)$.

In \cite{Gud10}, Gudmundsson studies certain families of Dyck paths, SYT, and pattern avoiding permutations. The main result in \cite{Gud10} related to our work is the following theorem for which the author provides a bijective proof.

\begin{theorem*}[\cite{Gud10}]
Let $d = k + p$. The class of Dyck paths of semilength $n$ that begin with at least $k$ successive up-steps, end with at least $p$ successive down-steps, and touch the $x$-axis at least once somewhere between the endpoints is equinumerous with the class of SYT of shape $(n, n-d)$.
\end{theorem*}

Here is a different connection with the same class of SYT.

\begin{proposition}\label{prop:n_n-d}
For $0\le d \le n$, Dyck paths of semilength $n+1$ having exactly $d+1$ returns are in bijection with SYT of shape $(n, n-d)$.
\end{proposition}

The bijection is defined as follows. Given a Dyck path of semilength $n+1$ with exactly $d+1$ returns, number each step from left to right ignoring the first up-step and skipping every down-step that touches the $x$-axis. Then create the SYT of shape $(n,n-d)$ by placing the labels of the $n$ up-steps in the first row and the labels of the $n-d$ labeled down-steps in the second row. For example:

\medskip
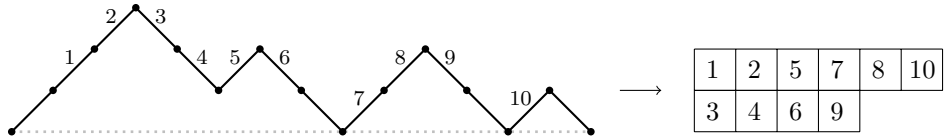
\begin{center}
\begin{tikzpicture}[scale=0.48]
\begin{scope}
\scriptsize
	\draw (1.4,1.8) node {1};
	\draw (2.4,2.8) node {2};
	\draw (3.6,2.8) node {3};
	\draw (4.6,1.8) node {4};
	\draw (5.4,1.8) node {5};
	\draw (6.6,1.8) node {6};
	\draw (8.4,0.8) node {7};
	\draw (9.4,1.8) node {8};
	\draw (10.6,1.8) node{9};
	\draw (12.3,0.8) node {10};
\draw[pathcolorlight] (0,0) -- (14,0);
\drawlinedots{0,1,2,3,4,5,6,7,8,9,10,11,12,13,14}%
{0,1,2,3,2,1,2,1,0,1,2,1,0,1,0}
\end{scope}

\draw[->] (14.5,1) -- (15.5,1);

\begin{scope}[xshift=162mm]
\small
\draw[step=1] (0,1) grid (6,2);
\draw[step=1] (0,0) grid (4,1);
\draw (0,1.5) node[right=1.3pt] {1};  
\draw (1,1.5) node[right=1.3pt] {2};  
\draw (2,1.5) node[right=1.3pt] {5};  
\draw (3,1.5) node[right=1.3pt] {7};  
\draw (4,1.5) node[right=1.3pt] {8};  
\draw (4.9,1.5) node[right=0.3pt] {10};  
\draw (0,0.5) node[right=1.3pt] {3};  
\draw (1,0.5) node[right=1.3pt] {4};  
\draw (2,0.5) node[right=1.3pt] {6};  
\draw (3,0.5) node[right=1.3pt] {9};  
\draw (6.5,1) node {.};
\end{scope}
\end{tikzpicture}
\end{center}
\medskip
We leave the checking that this map is indeed bijective as a nice exercise.

As a final example involving SYT with 3 or fewer rows, by placing markings on peaks $\U\D$ (or valleys $\D\U$) of the Dyck paths, we obtain a class enumerated by the large (resp.\ small) Schr\"oder numbers. A similar result involving Schr\"oder paths can be found in Pechenik~\cite{Pec14}.

\begin{proposition}\label{prop:Schr_marked}
The number of Dyck paths of semilength $n$ with $k$ marked peaks (resp.\ valleys) equals the number of tableaux of shape $(n,n)$ with label set $\{1,\ldots,2n-k\}$ such that the rows are strictly increasing and the columns are weakly (resp.\ strictly) increasing.\end{proposition}

Adding over $k$, these are known to be enumerated by the large Schr\"oder numbers \oeis{A006318} (resp.\ small Schr\"oder numbers \oeis{A001003}).  In the first case, where the columns are weakly increasing, the tableau is the transpose of a semistandard Young tableau. In the second case, where the columns are strictly increasing, such a tableau is called an \emph{increasing tableau}.

As in the classical bijection, we read our Dyck path from left-to-right, and insert a box in the first row for an unmarked up step and in the second row for an unmarked down step. When we encounter a marked peak or valley, we insert a box in both rows simultaneously. For example:

\medskip
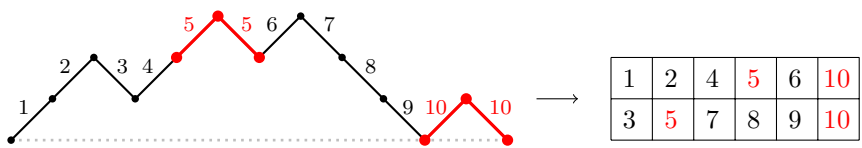
\begin{center}
\begin{tikzpicture}[scale=0.48]
	\begin{scope}
	\scriptsize
	\draw (0.7,0.8) node[left=1pt] {1};
	\draw (1.7,1.8) node[left=1pt] {2};
	\draw (2.3,1.8) node[right=1pt] {3};
	\draw (3.7,1.8) node[left=1pt] {4};
	\draw (4.7,2.8) node[left=1pt,color1] {\cbf{5}};
	\draw (5.3,2.8) node[right=1pt,color1] {\cbf{5}};
	\draw (6.7,2.8) node[left=1pt] {6};
	\draw (7.3,2.8) node[right=1pt] {7};
	\draw (8.3,1.8) node[right=1pt] {8};
	\draw (9.2,0.8) node[right=1pt] {9};
	\draw (10.8,0.8) node[left=1pt,color1] {\cbf{10}};
	\draw (11.3,0.8) node[right=1pt,color1] {\cbf{10}};
	\draw[pathcolorlight] (0,0) -- (12,0);
	\drawlinedots{0,1,2,3,4,5,6,7,8,9,10,11,12}%
	{0,1,2,1,2,3,2,3,2,1,0,1,0}
	\drawlinedots[pathmarked][nodemarked]{4,5,6}{2,3,2}
	\drawlinedots[pathmarked][nodemarked]{10,11,12}{0,1,0}
	\end{scope}
	
	\draw[->] (12.7,1) -- (13.7,1);
	
	\begin{scope}[xshift=145mm]
	\small
	\draw[step=1] (0,0) grid (6,2);
	\draw (0,1.5) node[right=1pt] {1};  
	\draw (0,0.5) node[right=1pt] {3};  
	\draw (1,1.5) node[right=1pt] {2};  
	\draw[color1] (1,0.5) node[right=1pt] {\cbf{5}};  
	\draw (2,1.5) node[right=1pt] {4};  
	\draw (2,0.5) node[right=1pt] {7};  
	\draw[color1] (3,1.5) node[right=1pt] {\cbf{5}};  
	\draw (3,0.5) node[right=1pt] {8};  
	\draw (4,1.5) node[right=1pt] {6};  
	\draw (4,0.5) node[right=1pt] {9};  
	\draw[color1] (4.85,1.5) node[right=1pt] {\cbf{10}};  
	\draw[color1] (4.85,0.5) node[right=1pt] {\cbf{10}};
	\draw (6.5,1) node {.};
	\end{scope}
\end{tikzpicture}
\end{center}
\medskip

To obtain tableaux with rows strictly increasing, we must avoid peaks at starting height zero. An alternative way to achieve this is is to use valleys instead, which never start at height zero. For example:

\medskip
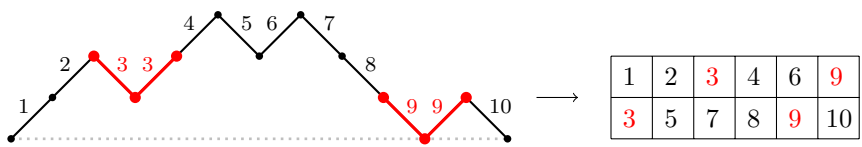
\begin{center}
\begin{tikzpicture}[scale=0.48]
	\begin{scope}
	\scriptsize
	\draw (0.7,0.8) node[left=1pt] {1};
	\draw (1.7,1.8) node[left=1pt] {2};
	\draw (2.3,1.8) node[right=1pt,color1] {\cbf{3}};
	\draw (3.7,1.8) node[left=1pt,color1] {\cbf{3}};
	\draw (4.7,2.8) node[left=1pt] {4};
	\draw (5.3,2.8) node[right=1pt] {5};
	\draw (6.7,2.8) node[left=1pt] {6};
	\draw (7.3,2.8) node[right=1pt] {7};
	\draw (8.3,1.8) node[right=1pt] {8};
	\draw (9.3,0.8) node[right=1pt,color1] {\cbf{9}};
	\draw (10.7,0.8) node[left=1pt,color1] {\cbf{9}};
	\draw (11.3,0.8) node[right=1pt] {10};
	\draw[pathcolorlight] (0,0) -- (12,0);
	\drawlinedots{0,1,2,3,4,5,6,7,8,9,10,11,12}%
	{0,1,2,1,2,3,2,3,2,1,0,1,0}
	\drawlinedots[pathmarked][nodemarked]{2,3,4}{2,1,2}
	\drawlinedots[pathmarked][nodemarked]{9,10,11}{1,0,1}
	\end{scope}
	
	\draw[->] (12.7,1) -- (13.7,1);
	
	\begin{scope}[xshift=145mm]
	\small
	\draw[step=1] (0,0) grid (6,2);
	\draw (0,1.5) node[right=1pt] {1};  
	\draw[color1] (0,0.5) node[right=1pt] {\cbf{3}};  
	\draw (1,1.5) node[right=1pt] {2};  
	\draw (1,0.5) node[right=1pt] {5};  
	\draw[color1] (2,1.5) node[right=1pt] {\cbf{3}};  
	\draw (2,0.5) node[right=1pt] {7};  
	\draw (3,1.5) node[right=1pt] {4};  
	\draw (3,0.5) node[right=1pt] {8};  
	\draw (4,1.5) node[right=1pt] {6};  
	\draw[color1] (4,0.5) node[right=1pt] {\cbf{9}};  
	\draw[color1] (5,1.5) node[right=1pt] {\cbf{9}};  
	\draw (4.85,0.5) node[right=1pt] {10};
	\draw (6.5,1) node {.};
	\end{scope}
\end{tikzpicture}
\end{center}

\medskip
It is clear that these maps have well-defined inverses.

\section{cm-Labeled Dyck paths to SYT}
\label{sec:nxpt}

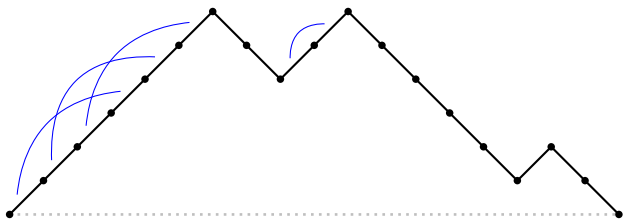
\begin{figure}[ht]
\begin{tikzpicture}[scale=0.45]
\begin{scope}
\draw[pathcolorlight] (-0.5,0) -- (17.5,0);
\drawlinedots{-0.5,0.5,1.5,2.5,3.5,4.5,5.5,6.5,7.5,8.5,9.5,10.5,11.5,12.5,13.5,14.5,15.5,16.5,17.5}%
{0,1,2,3,4,5,6,5,4,5,6,5,4,3,2,1,2,1,0}
\end{scope}
\begin{scope}[scale=1.44, color=color2]
\begin{scope}[xshift=-9mm, yshift=-3mm, rotate=45]
\draw[thick] (1,0) parabola bend (2.5,0.6) (4,0);
\draw[thick] (2,0) parabola bend (3.5,0.8) (5,0);
\draw[thick] (3,0) parabola bend (4.5,0.6) (6,0);
\end{scope}
\begin{scope}[xshift=47mm, yshift=25mm, rotate=45]
\draw[thick] (1,0) parabola bend (1.5,0.25) (2,0);
\end{scope}
\end{scope}
\end{tikzpicture}
\caption{A cm-labeled Dyck path.}
\label{fig:cm-labeled_dyck}
\end{figure}

As already mentioned, it is well-known that the \emph{full} set of Dyck paths of semilength $n$ is in bijection with the set of SYT of shape $(n,n)$, and we have seen several bijections from classes of Dyck paths to classes of SYT.  Focusing now on the \emph{full} set of SYT with $n$ boxes and no shape restriction, in this section we address the following question: Is there a class of Dyck paths that is in bijection with the set of SYT with $n$ boxes? Our answer, which is summarized in Corollary~\ref{cor:dyck_syt}, involves labeled Dyck paths, connected matchings, noncrossing partitions and nonnesting partitions.

We start by describing the combinatorial objects involved in our results. A graph on the set $[n]=\{1,2,\ldots,n\}$ is a \emph{partial matching} if every vertex has degree at most one. We will also refer to such graphs as \emph{involutions} since they are clearly in bijection with self-inverse permutations of $[n]$. We will call vertices of degree zero \emph{singletons}. A partial matching is a \emph{perfect matching} if every vertex has degree exactly one; note that the existence of a perfect matching implies that $n$ is even.  We will represent partial matchings by graphs on the number line with the edges drawn as arcs, with these arcs always drawn above the number line, as in Fig.~\ref{fig:partial_matchings}. A partial matching is a \emph{connected matching} if these arcs together with the $n$ points on the number line form a connected set as a subset of the plane.  For example, in Fig.~\ref{fig:partial_matchings}, the matching on the left is connected whereas the matching on the right has four connected components. Note that a partial matching on $[n]$ with $n > 1$ can only be connected if it is a perfect matching.  When $n=1$, we consider its unique partial matching (consisting of no arcs) to be connected.  

\begin{figure}[ht]
\begin{tikzpicture}[scale=0.55]
\begin{scope}
\foreach \x in {1,2,3,4,5,6,7,8}
\draw[fill=black] (\x,0) circle (0.08);
\foreach \x in {1,2,3,4,5,6,7,8}
\draw (\x,-0.4) node {\scriptsize \x};
\draw[thick] (1,0) parabola bend (3,0.8) (5,0);
\draw[thick] (2,0) parabola bend (5,1.2) (8,0);
\draw[thick] (3,0) parabola bend (4.5,0.6) (6,0);
\draw[thick] (4,0) parabola bend (5.5,0.6) (7,0);
\end{scope}
\begin{scope}[xshift=90mm]
\foreach \x in {1,2,3,4,5,6,7,8}
\draw[fill=black] (\x,0) circle (0.08);
\foreach \x in {1,2,3,4,5,6,7,8}
\draw (\x,-0.4) node {\scriptsize \x};
\draw[thick] (1,0) parabola bend (3.5,1) (6,0);
\draw[thick] (2,0) parabola bend (2.5,0.2) (3,0);
\draw[thick] (5,0) parabola bend (6,0.4) (7,0);
\end{scope}
\end{tikzpicture}
\caption{The partial matchings $(1\,5)(2\,8)(3\,6)(4\,7)$ and $(1\,6)(2\,3)(4)(5\,7)(8)$.}
\label{fig:partial_matchings}
\end{figure}
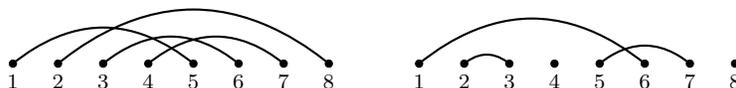

\begin{definition*}
A \emph{cm-labeled Dyck path} is a Dyck path where each $k$-ascent is labeled by a connected matching on $[k]$ (see Fig.~\ref{fig:cm-labeled_dyck} for an example).
\end{definition*}

First note there are no connected matchings on $[k]$ when $k$ is odd and greater than 1, so all the ascents in a cm-labeled Dyck path must be of even length or length 1. Secondly, a cm-labeled Dyck path all of whose ascents are length 1, 2 or 4 is equivalent to its unlabeled version since there is a unique connected matching on $[k]$ when $k = 1,2, 4$. The first interesting case is when a Dyck path has $6$-ascents, because then there are 4 ways to label each 6-ascent:
\medskip
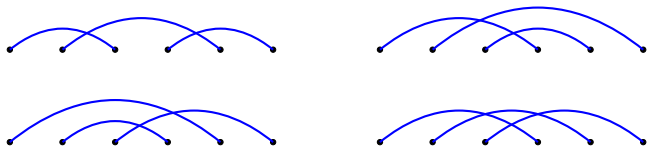
\begin{center}
\begin{tikzpicture}[scale=0.6]
\begin{scope}
\draw[color2,thick] (1,0) parabola bend (2,0.4) (3,0);
\draw[color2,thick] (2,0) parabola bend (3.5,0.6) (5,0);
\draw[color2,thick] (4,0) parabola bend (5,0.4) (6,0);
\foreach \x in {1,2,3,4,5,6}
\draw[fill] (\x,0) circle (0.05);
\end{scope}
\begin{scope}[xshift=200pt]
\draw[color2,thick] (1,0) parabola bend (2.5,0.6) (4,0);
\draw[color2,thick] (2,0) parabola bend (4,0.8) (6,0);
\draw[color2,thick] (3,0) parabola bend (4,0.4) (5,0);
\foreach \x in {1,2,3,4,5,6}
\draw[fill] (\x,0) circle (0.05);
\end{scope}
\begin{scope}[yshift=-50pt]
\draw[color2,thick] (1,0) parabola bend (3,0.8) (5,0);
\draw[color2,thick] (2,0) parabola bend (3,0.4) (4,0);
\draw[color2,thick] (3,0) parabola bend (4.5,0.6) (6,0);
\foreach \x in {1,2,3,4,5,6}
\draw[fill] (\x,0) circle (0.05);
\end{scope}
\begin{scope}[xshift=200, yshift=-50pt]
\draw[color2,thick] (1,0) parabola bend (2.5,0.6) (4,0);
\draw[color2,thick] (2,0) parabola bend (3.5,0.6) (5,0);
\draw[color2,thick] (3,0) parabola bend (4.5,0.6) (6,0);
\foreach \x in {1,2,3,4,5,6}
\draw[fill] (\x,0) circle (0.05);
\draw (6.8,1) node {.};
\end{scope}
\end{tikzpicture}
\end{center}

\medskip
In a partial matching, two arcs $(i,j)$ and $(k,\ell)$ form a \emph{crossing} if $i<k<j<\ell$ or, equivalently, if the arcs cross in the graphical representation of the partial matching.  A \emph{$k$-crossing} is a set of $k$ arcs in a partial matching $M$ that are pairwise crossing, and the \emph{crossing number} of $M$ is the largest $k$ such that $M$ has a $k$-crossing. A partial matching is \emph{$k$-noncrossing} if it has no $k$-crossings. For example, the partial matching $(1\,5)(2\,8)(3\,6)(4\,7)$ on the left in Fig.~\ref{fig:partial_matchings} is 4-noncrossing and has crossing number 3 due to the arcs $(1\,5)(3\,6)(4\,7)$.

Analogously, two arcs $(i,j)$ and $(k,\ell)$ form a \emph{nesting} if $i<k<\ell<j$. A \emph{$k$-nesting} is a set of $k$ arcs in a partial matching that are pairwise nesting, with the \emph{nesting number} and \emph{$k$-nonnesting} defined in a way parallel to the analogous terms for crossings. For example, the partial matching $(1\,5)(2\,8)(3\,6)(4\,7)$ above is 3-nonnesting and has nesting number 2 due to the arcs $(2\,8)(3\,6)$ or $(2\,8)(4\,7)$.

Our bijection relies heavily on the following result of Burrill et al. \cite[Proposition~12]{BCFMM16}.

\begin{proposition}[\cite{BCFMM16}]
The following classes are in bijection:
\begin{enumerate}[$(i)$]
\item the set of $k$-noncrossing partial matchings on $[n]$ with $s$ singletons;
\item the set of $k$-nonnesting partial matchings on $[n]$ with $s$ singletons;
\item the set of involutions on $[n]$ with decreasing subsequences of length at most $2k-1$ and with $s$ fixed points;
\item the set of SYT with $n$ boxes, at most $2k-1$ rows, and $s$ odd columns.
\end{enumerate}
\end{proposition}

Using this result together with one of the standard bijections between Dyck paths and noncrossing partitions, we arrive at the following:
\begin{proposition}\label{prop:nxpt}
The number of cm-labeled Dyck paths of semilength $n$ with $s$ singletons and $k$-noncrossing labels equals the number of SYT with $n$ boxes, $s$ columns of odd length, and at most $2k-1$ rows.
\end{proposition}

Let $\syt{n}$ denote the number of SYT with $n$ boxes (cf.\ \oeis{A000085}). Letting $k$ be sufficiently large and summing over $s$, Proposition~\ref{prop:nxpt} yields:

\begin{corollary}\label{cor:dyck_syt}
The number of cm-labeled Dyck paths of semilength $n$ equals $\syt{n}$.
\end{corollary}

The proof of Proposition~\ref{prop:nxpt} consists of several bijective steps: from cm-labeled Dyck paths to $k$-noncrossing partial matchings to $k$-nonnesting partial matchings to involutions, and finally to SYT via the Robinson--Schensted--Knuth (RSK) algorithm. We proceed to illustrate this elaborate construction by means of an example.

Consider the cm-labeled Dyck path $D$ depicted in Fig.~\ref{fig:cm-labeled_dyck}. Number the up-steps in the following fashion. First number the down-steps with $\{1,\dots,9\}$ in increasing order from left-to-right. Then move each such label horizontally to the left until it meets its corresponding up-step, resulting in a labeling on the up-steps:

\medskip
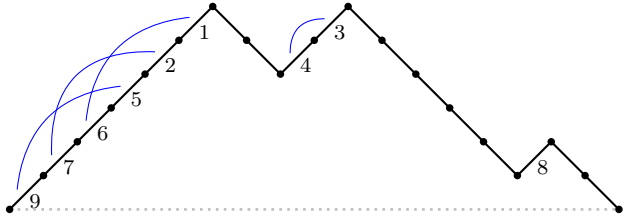
\begin{center}
\begin{tikzpicture}[scale=0.45]
\begin{scope}
\scriptsize
	\draw (5.25,5.25) node {1};
	\draw (4.25,4.25) node {2};
	\draw (3.25,3.25) node {5};
	\draw (2.25,2.25) node {6};
	\draw (1.25,1.25) node {7};
	\draw (0.25,0.25) node {9};
	\draw (9.25,5.25) node {3};
	\draw (8.25,4.25) node {4};
	\draw (15.25,1.25) node {8};
\draw[pathcolorlight] (-0.5,0) -- (17.5,0);
\drawlinedots{-0.5,0.5,1.5,2.5,3.5,4.5,5.5,6.5,7.5,8.5,9.5,10.5,11.5,12.5,13.5,14.5,15.5,16.5,17.5}%
{0,1,2,3,4,5,6,5,4,5,6,5,4,3,2,1,2,1,0}
\draw (19,3) node {.};
\end{scope}
\begin{scope}[scale=1.44, color=color2]
\begin{scope}[xshift=-9mm, yshift=-3mm, rotate=45]
\draw[thick] (1,0) parabola bend (2.5,0.6) (4,0);
\draw[thick] (2,0) parabola bend (3.5,0.8) (5,0);
\draw[thick] (3,0) parabola bend (4.5,0.6) (6,0);
\end{scope}
\begin{scope}[xshift=47mm, yshift=25mm, rotate=45]
\draw[thick] (1,0) parabola bend (1.5,0.25) (2,0);
\end{scope}
\end{scope}
\end{tikzpicture}
\end{center}

The partial matching $M_{D}$ associated with $D$ is obtained by applying the connected matching on each ascent to the ascent's numbers; see Fig.~\ref{fig:oscillating_bijection}. 

\medskip
\begin{figure}[ht]
\begin{tikzpicture}[scale=0.6]
\foreach \x in {1,2,3,4,5,6,7,8,9}
\draw[fill=black] (\x,0) circle (0.08);
\foreach \x in {1,2,3,4,5,6,7,8,9}
\draw (\x,-0.4) node {\scriptsize \x};
\draw[thick] (1,0) parabola bend (3.5,1) (6,0);
\draw[thick] (2,0) parabola bend (4.5,1) (7,0);
\draw[thick] (3,0) parabola bend (3.5,0.2) (4,0);
\draw[thick] (5,0) parabola bend (7,0.8) (9,0);
\end{tikzpicture}
\caption{The partial matching $M_D=(1\,6)(2\,7)(3\,4)(5\,9)(8)$.}
\label{fig:oscillating_bijection}
\end{figure}
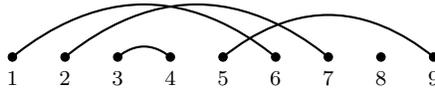

Note that if the cm-labels are $k$-noncrossing, so is the partial matching, and clearly the number of singletons on the Dyck path equals the number of singletons in the matching. Also, the connected components of $M_D$ correspond to the ascents of $D$.

The step from $k$-noncrossing to $k$-nonnesting partial matchings works by modifying a known bijection between perfect matchings and oscillating tableaux. We follow a technique from \cite{C+07} (see also \cite{BCFMM16}) by first mapping a partial matching to an oscillating tableau, then transposing the tableau, and then mapping the result back to a partial matching. The modified map is constructed so as to preserve the number of singletons. We restrict our attention to \emph{weakly oscillating tableau} of empty shape and length $n$, that is, a sequence of partitions $\Lambda = (\lambda^0, \lambda^1, \ldots, \lambda^n)$ such that:
\begin{enumerate}[$(i)$]
\item $\lambda^0 = \lambda^n = \emptyset$, the empty partition;
\item for $1\leq i \leq n$, $\lambda^i$ is obtained from $\lambda^{i-1}$ by either doing nothing, adding a box, or deleting a box.  
\end{enumerate}
Given a partial matching $M$ on $[n]$, represented as a graph on the number line, we construct a sequence of tableaux $T^n,\dots, T^0$ as follows. We begin by setting $T^n=\emptyset$, the empty tableau. For $n \geq j \geq 1$, construct $T^{j-1}$ according to the following rules.
\begin{enumerate}
\item If $j$ is a singleton in $M$, then set $T^{j-1}=T^j$.
\item If $j$ is the right-hand endpoint of an arc $(i,j)$ in $M$, then RSK insert\footnote{%
See~\cite[\S3.1]{Sag01} or \cite[\S7.11]{Sta99} for an introduction to RSK insertion.} $i$ into $T^j$.
\item If $j$ is the left-hand endpoint of an arc $(j,k)$ in $M$, then remove $j$ (and the box that contained $j$) from $T^j$.  
\end{enumerate}

For the partial matching $M_D = (1\,6)(2\,7)(3\,4)(5\,9)(8)$, the sequence $T^0, \dots, T^9$ and the resulting weakly oscillating tableau $\Lambda = (\lambda^0, \ldots, \lambda^9)$ are displayed on Table~\ref{tab:oscillating_bijection}. Recall that the construction of the $T^j$ proceeds from right to left, and that $T^{j-1}$ is determined by the properties of the number $j$, rather than of $j-1$.  

\begin{table}[ht]
\small
\Yvcentermath1
\[
\begin{array}{c|@{\hspace{2ex}}cccccccccc}
j & 0 & 1 & 2 & 3 & 4 & 5 & 6 & 7 & 8 & 9 \\[2pt] \hline \\[-4pt]
T^j & \emptyset & \young(1) & \young(1,2) & \young(13,2) & \young(1,2) & \young(1,2,5) & \young(2,5) & \young(5) & \young(5) & \emptyset \\[4ex]
\lambda^j & \emptyset & \yng(1) & \yng(1,1) & \yng(2,1) & \yng(1,1) & \yng(1,1,1) & \yng(1,1) & \yng(1) & \yng(1) & \emptyset
\end{array}
\]
\smallskip
\caption{Sequences corresponding to $M_D = (1\,6)(2\,7)(3\,4)(5\,9)(8)$.}
\label{tab:oscillating_bijection}
\end{table}

Conversely, given an oscillating tableau $\Lambda= (\lambda^0, \ldots, \lambda^n)$, set $(T^0, M^0) = (\emptyset, \emptyset)$ and, for $1 \leq j \leq n$, construct $(T^j, M^j)$ from left to right according to the following rules:
\begin{enumerate}
\item If $\lambda^j=\lambda^{j-1}$, then set $(T^j, M^j) = (T^{j-1}, M^{j-1})$.
\item If $\lambda^j \subset \lambda^{j-1}$, then obtain $T^j$ from $T^{j-1}$ by reverse RSK insertion, starting with the entry $k$ in the box in position $\lambda^j \setminus \lambda^{j-1}$. This will result in an entry $i \leq k$ leaving $T^{j-1}$. Add the pair $(i,j)$ to $M^{j-1}$ to obtain $M^j$.
\item If $\lambda^j \supset \lambda^{j-1}$, let $T^j$ be obtained from $T^{j-1}$ by adding the box $\lambda^j\setminus \lambda^{j-1}$ with entry $j$, and simply let $M^j = M^{j-1}$.
\end{enumerate}
The image of $\Lambda$ is then the partial matching $M^n$.

With this bijection in place, the composite bijection from $k$-noncrossing partial matchings to $k$-nonnesting partial matchings is given by
\[ M \mapsto \Lambda \mapsto \Lambda^t \mapsto \widehat{M}, \]
where $\Lambda^t := \big((\lambda^0)^t, \ldots, (\lambda^n)^t\big)$ is the weakly oscillating tableau obtained by transposing the partitions from $\Lambda$, and $\widehat{M}$ is the partial matching resulting from the inverse map above applied to $\Lambda^t$.

We leave it as an exercise for the reader to check that 
\[ \widehat{M_D} = (1\,9)(2\,4)(3\,7)(5\,6)(8). \] 
Observe in Fig.~\ref{fig:crossing_nesting} that $M_D$ has a 3-crossing and a 2-nesting, whereas $\widehat{M_D}$ has a 2-crossing and a 3-nesting (cf.\ \cite[Thm.~3.2]{C+07}). Moreover, they both have the same number of singletons. 

\begin{figure}[ht]
\begin{tikzpicture}[scale=0.5]
\begin{scope}
\foreach \x in {1,2,3,4,5,6,7,8,9}
\draw[fill=black] (\x,0) circle (0.08);
\foreach \x in {1,2,3,4,5,6,7,8,9}
\draw (\x,-0.4) node {\scriptsize \x};
\draw[thick] (1,0) parabola bend (3.5,1) (6,0);
\draw[thick] (2,0) parabola bend (4.5,1) (7,0);
\draw[thick] (3,0) parabola bend (3.5,0.2) (4,0);
\draw[thick] (5,0) parabola bend (7,0.8) (9,0);
\end{scope}
\begin{scope}[xshift=28em]
\foreach \x in {1,2,3,4,5,6,7,8,9}
\draw[fill=black] (\x,0) circle (0.08);
\foreach \x in {1,2,3,4,5,6,7,8,9}
\draw (\x,-0.4) node {\scriptsize \x};
\draw[thick] (1,0) parabola bend (5,1.2) (9,0);
\draw[thick] (2,0) parabola bend (3,0.4) (4,0);
\draw[thick] (3,0) parabola bend (5,0.8) (7,0);
\draw[thick] (5,0) parabola bend (5.5,0.2) (6,0);
\end{scope}
\end{tikzpicture}
\caption{The associated partial matchings $M_D$ and $\widehat{M_D}$.}
\label{fig:crossing_nesting}
\end{figure}
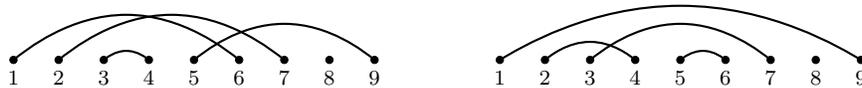

Treating the arcs as transpositions, partial matchings $M$ on $[n]$ are trivially in one-to-one correspondence with involutions $\pi$ on $[n]$, and the number of singletons in $M$ clearly equals the number of fixed points of $\pi$. For example, $\widehat{M_D} = (1\,9)(2\,4)(3\,7)(5\,6)(8)$ corresponds to the involution $\pi_D=947265381$. Now, using the RSK algorithm on the involution $\pi_D$, we finally get the SYT $T_D$ corresponding to the cm-labeled Dyck path $D$ (see Fig.~\ref{fig:syt}).

\medskip
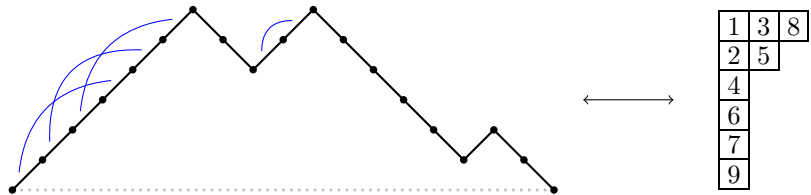
\begin{figure}[ht]
\begin{tikzpicture}[scale=0.4]
\begin{scope}
\draw[pathcolorlight] (-0.5,0) -- (17.5,0);
\drawlinedots{-0.5,0.5,1.5,2.5,3.5,4.5,5.5,6.5,7.5,8.5,9.5,10.5,11.5,12.5,13.5,14.5,15.5,16.5,17.5}%
{0,1,2,3,4,5,6,5,4,5,6,5,4,3,2,1,2,1,0}
\end{scope}
\begin{scope}[scale=1.44, color=color2]
\begin{scope}[xshift=-9mm, yshift=-3mm, rotate=45]
\draw[thick] (1,0) parabola bend (2.5,0.6) (4,0);
\draw[thick] (2,0) parabola bend (3.5,0.8) (5,0);
\draw[thick] (3,0) parabola bend (4.5,0.6) (6,0);
\end{scope}
\begin{scope}[xshift=47mm, yshift=25mm, rotate=45]
\draw[thick] (1,0) parabola bend (1.5,0.25) (2,0);
\end{scope}
\end{scope}
\begin{scope}[xshift=48em]
\small
\draw[<->] (1,3) -- (3.2,3); 
\node at (6,3){$\young(138,25,4,6,7,9)$}; 
\end{scope}
\end{tikzpicture}
\caption{The cm-labeled Dyck path $D$ and its corresponding SYT $T_D$.}
\label{fig:syt}
\end{figure}

\section{Dyck paths to SYT of rectangular shape} 
\label{sec:rectangular}

In a recent paper, Wettstein \cite{Wett17} discussed certain sets of balanced bracket expressions that are enumerated by the $d$-dimensional Catalan numbers, and he introduced a class of prime elements that serve as building blocks for the entire set of such expressions. If $C_d(x)$ and $P_d(x)$ are the generating functions for these sets (with $C_d(0)=P_d(0)=1$), respectively, Wettstein proved the relation
\begin{equation}\label{eq:d-Catalan}
  C_d(x) = P_d\left(xC_d(x)^d\right) \,\text{ for every $d\ge 2$.}
\end{equation} 
By \cite[Example~14]{BGW18}, this means that $C_d(x)$ is the $d$-th noncrossing partition transform of $P_d(x)$ (cf.\ Subsection~\ref{subsec:noncrossing}), which provides a way to bijectively connect SYT of shape $(n^d)$ with Dyck paths. 

We proceed to elaborate on this bijective connection.
\begin{definition}
For $d,n\in \mathbb{N}$ and $d\ge 2$, let $\mathcal{W}_d(n)$ be the set of words $w$ of length $d\!\cdot\!n$ over the alphabet $\{a_1,a_2,\dots,a_d\}$ with $\#(w,a_1)=\cdots=\#(w,a_d)$, and such that for every prefix $u$ of $w$ we have
\[ \#(u,a_1)\ge \#(u,a_2) \ge \cdots \ge \#(u,a_d), \]
where $\#(z,\ell)$ denotes the number of times the letter $\ell$ appears in the word $z$. Further let $\widetilde{\mathcal{W}}_d(n)$ be the set of corresponding primitives (factor-free) words, i.e.\ words in $\mathcal{W}_d(n)$ that do not contain any nonempty contiguous subword in $\mathcal{W}_d(j)$ for $j<n$.
\end{definition}

\begin{proposition}
$\mathcal{W}_d(n)$ is in bijection with the set of SYT of shape $(n^d)$ and their elements are enumerated by the $d$-dimensional Catalan numbers. Moreover, by \cite[Lemma~4.3]{Wett17}, the set $\widetilde{\mathcal{W}}_d(n)$ of primitive elements is enumerated by the function $P_d(x)$ satisfying \eqref{eq:d-Catalan}.
\end{proposition}

The bijection between $\mathcal{W}_d(n)$ and the set of SYT of shape $(n^d)$ is simple and well-known, and is given at the start of the proof of Theorem~\ref{thm:dCatalan} below.

\begin{example}[$d=3$]
The 3-dimensional Catalan numbers \oeis{A005789} are given by
\[ 1, 1, 5, 42, 462, 6006, 87516, 1385670, 23371634, 414315330, \dots \]
and the corresponding coefficients of $P_3(x)$ are (cf.\ \oeis{A268538})
\[ 1, 1, 2, 12, 107, 1178, 14805, 203885, 3002973, 46573347, \dots \]
For example, using the alphabet $\{a,b,c\}$ we have that $\mathcal{W}_3(2)$ consists of the five words
\[ aabcbc\quad ababcc\quad abcabc\quad aabbcc\quad abacbc \]
that correspond to the SYT

{\small
\begin{equation*}
\Yvcentermath1
\young(12,35,46)\qquad \young(13,24,56)\qquad \young(14,25,36)\qquad \young(12,34,56)\qquad   \young(13,25,46) \ \ .
\end{equation*}
}
Note that $aabbcc$ and $abacbc$ are the only words in $\widetilde{\mathcal{W}}_3(2)$ (primitives of length 6). In fact, the other three words $a\cbfit{abc}bc$, $ab\cbfit{abc}c$, and $abc\cbfit{abc}$ all contain the factor $abc$ as a subword. On the other hand, the set $\widetilde{\mathcal{W}}_3(3)$ consists of 12 elements:
\begin{gather*}
 aaabbbccc\quad aaabbcbcc\quad aababbccc\quad aabbacbcc\quad aabacbbcc\quad aabbaccbc \\
 aabbcacbc\quad abacabbcc\quad abacbacbc\quad abaacbbcc\quad abaacbcbc\quad ababaccbc
\end{gather*} 
corresponding to
 
{\small
\begin{gather*}
\Yvcentermath1
 \young(123,456,789)\qquad  \young(123,457,689)\qquad  \young(124,356,789)\qquad  \young(125,347,689)\qquad  \young(124,367,589)\qquad  \young(125,348,679) \ \ \ \\[3pt] 
\Yvcentermath1
 \young(126,348,579)\qquad  \young(135,267,489)\qquad  \young(136,258,479)\qquad  \young(134,267,589)\qquad  \young(134,268,579)\qquad  \young(135,248,679) \ \ .
\end{gather*}
}
\end{example}

\begin{theorem} \label{thm:dCatalan}
The set of SYT of shape $(n^d)$ is in bijection with the set $\mathfrak{D}_n^{\bf p}(d,0)$ of Dyck paths of semilength $d\!\cdot\!n$ created from strings of the form $\D$ and $\U^{d\cdot j}\D$ for $j=1,\ldots, n$, and such that each $d\!\cdot\!j$-ascent may be labeled in $p_j$ different ways. Here $\mathbf{p}=(p_n)$ denotes the sequence of coefficients of $P_d(x)$, called $d$-dimensional prime Catalan numbers in \cite{Wett17}.
\end{theorem}
\begin{proof}
This could be proved using \eqref{eq:d-Catalan} together with results from \cite{BGMW16,BGW18}, but here we give a bijective proof by example.  Consider the SYT

\Yvcentermath1
\Yboxdim14pt
\def\X{10}\def\XI{11}\def\XII{12}
\[ T=\; \young(1368,249\XI,57\X\XII)\ . \]

\medskip
Going through the numbers 1 through 12 in $T$, we write (left to right) $a$, $b$, or $c$ if the number is on the first, second, or third row of the tableaux, respectively. This gives us the word $w_T = ababcacabcbc$.

Moving now from right to left, we extract the factors of $w_T$ through the reduction
\[ ababcac\cbfit{abc}bc \underset{\cbfit{abc}}{\longrightarrow} ababcacbc 
  = ab\cbfit{abc}acbc \underset{\cbfit{abc}}{\longrightarrow} \cbfit{abacbc}, \]
which gives us the factors $abacbc$, $abc$, and $abc$. We record the length of the most left factor in each reduction step: $\ell_1(ababcac)=7$ and $\ell_2(ab)=2$ and construct a Dyck path
\[ D_T = \U^6 \D^{\ell_2} \U^3 \D^{\ell_1-\ell_2} \U^3 \D^{12-\ell_1} = \U^{6}\D^2 \U^3 \D^5 \U^3 \D^5, \]
where the ascents are labeled (from left to right) by the primitive words $abacbc$, $abc$, and $abc$, see Fig.~\ref{fig:DyckWord}. This process is clearly reversible.
\end{proof}

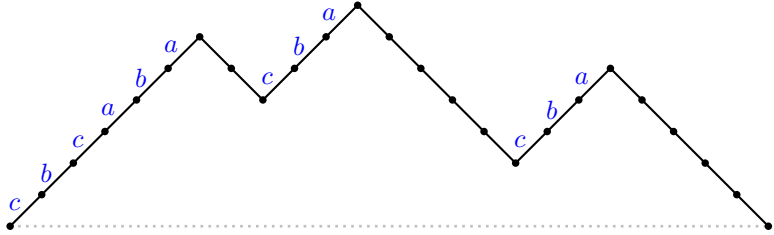
\begin{figure}[ht]
\begin{tikzpicture}[scale=0.4]
\small
\color{color2}
	\node [left=1pt] at (5.2,5.7) {$\mathbfit{a}$};
	\node [left=1pt] at (4.2,4.7) {$\mathbfit{b}$};
	\node [left=1pt] at (3.2,3.7) {$\mathbfit{a}$};
	\node [left=1pt] at (2.2,2.7) {$\mathbfit{c}$};
	\node [left=1pt] at (1.2,1.7) {$\mathbfit{b}$};
	\node [left=1pt] at (0.2,0.7) {$\mathbfit{c}$};
	\node [left=1pt] at (10.2,6.7) {$\mathbfit{a}$};
	\node [left=1pt] at (9.2,5.7) {$\mathbfit{b}$};
	\node [left=1pt] at (8.2,4.7) {$\mathbfit{c}$};
	\node [left=1pt] at (18.2,4.7) {$\mathbfit{a}$};
	\node [left=1pt] at (17.2,3.7) {$\mathbfit{b}$};
	\node [left=1pt] at (16.2,2.7) {$\mathbfit{c}$};
\draw[pathcolorlight] (-0.5,0) -- (23.5,0);
\drawlinedots{-0.5,0.5,1.5,2.5,3.5,4.5,5.5,6.5,7.5,8.5,9.5,10.5,11.5,12.5,13.5,14.5,15.5,16.5,17.5,18.5,19.5,20.5,21.5,22.5,23.5}%
{0,1,2,3,4,5,6,5,4,5,6,7,6,5,4,3,2,3,4,5,4,3,2,1,0}
\end{tikzpicture}
\caption{Labeled Dyck path $D_T$ associated with $w_T = ababcacabcbc$.}
\label{fig:DyckWord}
\end{figure}

\begin{remark}
For $d=3$ this offers a different Dyck path representation from the one given in Proposition~\ref{prop:nnn_syt}.
\end{remark}

\section{Labeled Motzkin paths to SYT}
\label{sec:Motzkin}

In Section~\ref{sec:nxpt} we considered adding extra structure to Dyck paths of length $2n$ to obtain objects equinumerous to SYT with $n$ boxes. In this section, we discuss other equinumerous sets which instead are obtained by adding extra structure to Motzkin paths of length $n$.

\begin{proposition}\label{prop:MotzkinPaths}
The following objects, defined by Motzkin paths of length $n$ with $s$ flat steps and some additional structure, are in bijection with partial matchings on $[n]$ having $s$ singletons and thus also with SYT with $n$ boxes and $s$ odd columns:
\begin{itemize}
\item \emph{Height-labeled Motzkin paths}, where each down-step starting at height $i$ is  given a label from $[i]$.
\item \emph{Full rook Motzkin paths}, which have rooks placed in their lower shape such that there is exactly one in the ``row" beneath each up-step and exactly one in the ``column" beneath each down-step, where ``row" and ``column" refer to the $45^\circ$ rotation.
\item \emph{Yamanouchi-colored Motzkin paths}
which can be defined by their correspondence with weakly oscillating tableaux. Up-steps, down-steps, and flat-steps correspond to adding, removing, or leaving as-is, respectively, and the label specifies the row in which to add or remove a box.
\end{itemize}
\end{proposition}

For instance, for the partial matching $(1\,6)(2\,7)(3\,4)(5\,9)(8)$ discussed in Section~\ref{sec:nxpt}, we have the labeled Motzkin paths in Fig.~\ref{fig:rookMotzkin}.

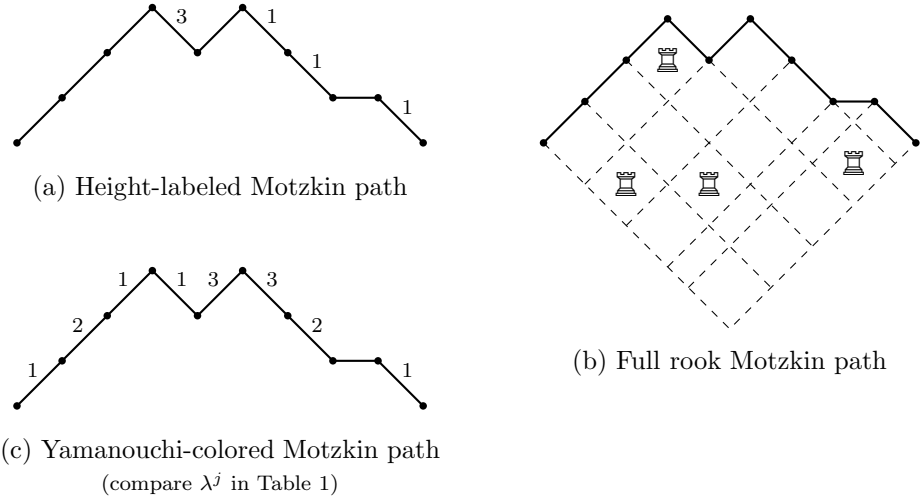
\begin{figure}[ht]
\begin{tikzpicture}
\scriptsize
\begin{scope}[scale=0.5]
\drawlinedots{0,1,2,3,4,5,6,7,8,9}{0,1,2,3,2,3,2,1,1,0}
\draw (3.65,2.8) node {3};
\draw (5.65,2.8) node {1};
\draw (6.65,1.8) node {1};
\draw (8.65,0.8) node {1};
\node at (4.5,-1) {\small (a) Height-labeled Motzkin path};
\end{scope}
\begin{scope}[xshift=60mm, scale=0.5]
\draw[dashed] (0,0) -- (4.5,-4.5) -- (9,0);
\draw[dashed] (1,1) -- (5.5,-3.5);
\draw[dashed] (2,2) -- (6.5,-2.5);
\draw[dashed] (4,2) -- (7.5,-1.5);
\draw[dashed] (7,1) -- (8.5,-0.5);
\draw[dashed] (1,-1) -- (4,2);
\draw[dashed] (2,-2) -- (6,2);
\draw[dashed] (3,-3) -- (7,1);
\draw[dashed] (3.5,-3.5) -- (8,1);
\drawlinedots{0,1,2,3,4,5,6,7,8,9}{0,1,2,3,2,3,2,1,1,0}
\draw (3,2) node[] {\large\rook};
\draw (2,-1) node[] {\large\rook};
\draw (4,-1) node[] {\large\rook};
\draw (7.5,-0.5) node[] {\large\rook};
\node at (4.5,-5.3) {\small (b) Full rook Motzkin path};
\end{scope}
\begin{scope}[yshift=-35mm,scale=0.5]
\drawlinedots{0,1,2,3,4,5,6,7,8,9}{0,1,2,3,2,3,2,1,1,0}
\draw (0.35,0.8) node {1};
\draw (1.35,1.8) node {2};
\draw (2.35,2.8) node {1};
\draw (3.65,2.8) node {1};
\draw (4.35,2.8) node {3};
\draw (5.65,2.8) node {3};
\draw (6.65,1.8) node {2};
\draw (8.65,0.8) node {1};
\node at (4.5,-1) {\small (c) Yamanouchi-colored Motzkin path};
\node[below=5pt] at (4.5,-1) {(compare $\lambda^j$ in Table~\ref{tab:oscillating_bijection})};
\end{scope}
\end{tikzpicture}
\caption{Motzkin paths corresponding to $(1\,6)(2\,7)(3\,4)(5\,9)(8)$.}
\label{fig:rookMotzkin}
\end{figure}

In contrast, the corresponding cm-labeled Dyck path is given in Fig.~\ref{fig:cm-labeled_dyck}.

Before proving Proposition~\ref{prop:MotzkinPaths}, let us put it in context with related results in the literature.  The bijection with height-labeled Motzkin paths is somewhat well known. The other two bijections are simple extensions of the better-known case when $s=0$. Height-labeled Motzkin paths are a case of the \emph{histoires} of orthogonal polynomials. This bijection is due to Fran\c{c}on and Viennot \cite{Fra78,FV79}. In the Dyck path case ($s=0$), height-labeled paths appear in Callan's survey of double factorials~\cite{Cal09} and are also called \emph{Hermite histoires}. Again for the case when $s=0$, full rook Motzkin paths are better known as full rook placements in Ferrers shapes. These were used by Krattenthaler~\cite{Kratt06} to extend the work of Chen et al.~\cite{C+07}. For a reader already familiar with Fomin growth diagrams, full rook Motzkin paths are a simple intermediate step in the bijection between height-labeled and Yamanouchi-colored Motzkin paths. Yamanouchi-colored Motzkin paths were introduced by Eu et al.~\cite{E+13}, who gave a definition and bijection using the language of Motzkin paths.

\begin{proof}[Proof of Proposition~\ref{prop:MotzkinPaths}]
First, there is a simple bijection between partial matchings and full rook Motzkin paths. Each pair $(i,j)$ in the matching with $i<j$ indicates an up-step at step $i$ and a down-step at step $j$. A singleton at $i$ indicates a flat-step at step $i$. We then draw the path from left to right according to these steps and place rooks at the positions determined by the matching, as in Fig.~\ref{fig:rookMotzkin}(b). For the reverse map, simply match the two steps diagonal from each rook, and leave the flats as singletons.

To make the bijection between height-labeled and full rook Motzkin paths easier to state, we use the terms ``row'' and ``column'' for the shape beneath the full rook path by considering the result of rotating it $45^\circ$ counterclockwise. We assign height-labels to each down-step starting at height $i$ (from left to right) according to the height of the rook in the column below, ignoring any rows with a rook in an earlier column. For example, in Fig.~\ref{fig:rookMotzkin}, the first down-step in (a) has label 3 because in (b) the rook is at height 3 in the column beneath this down-step.  A more interesting case is the third column, where the down-step has label 1 because it has a column of four beneath it, but ignoring the rows with the rooks already placed, there are two places available and the rook is in the first.  Observe that the number of places available is always the starting height of the down-step, so we do indeed arrive at a height-labeled Motzkin path. Clearly, this map is easily reversed.
	
Finally, we defined Yamanouchi-colored Motzkin paths by their correspondence with weakly oscillating tableaux, so the bijection with partial matchings is simply the one we have already seen in Section~\ref{sec:nxpt}. 
\end{proof}

\section{Further remarks}
\label{sec:further_remarks}

\subsection{Dyck paths to restricted set partitions}

Many of our connections between Dyck paths and Young tableaux involved either the map $\varphi:\Dyck(n)\to\Part_\nmi(n)$ (as in Sections~\ref{sec:flag_shape} and~\ref{sec:height3}) or the classic bijection from Dyck paths to noncrossing partitions (which appeared implicitly in Section~\ref{sec:nxpt}). We can describe the reverse map for both bijections in the same way, as in Equation~\eqref{eq:Dyck_from_Part}. It is then straightforward to generalize as follows.
	
\begin{proposition}
A map $\Dyck(n) \to \Part(n)$ is injective if, for each $k$-ascent in the Dyck path followed by the $m$th down step, there is a $k$-block of the partition with minimum $m$.
\end{proposition}

For example, such a map would take the Dyck path $\U^4\D^3\U^2\D\U\D\U^2\D^4$ to a partition of the form 1$\square\square\square|4\square|5|6\square$, because there are four ascents of sizes 4, 2, 1, and 2 followed by the 1st, 4th, 5th, and 6th down steps, respectively. If we place the remaining elements 23789 into blocks greedily from left to right, we obtain the nomincreasing partition $1237|48|5|69$. If instead we place elements greedily into blocks from right to left, preserving the minimum element of each block, we obtain the partition $1239|48|5|67$. It is not hard to see that these two maps are the map $\varphi$ to nomincreasing partitions and the classic map to noncrossing partitions.	
	
Another way to describe the relationship between these maps is to look at the \textit{front representations} of partitions studied by Kim~\cite{Kim11}. In this case, a block $B=\{a_1,\ldots,a_k\}$ with $a_1<\cdots<a_k$ is associated with the arc diagram $(a_1,a_2),(a_1,a_3),\ldots,(a_1,a_k)$ instead of $(a_1,a_2),(a_2,a_3),\ldots,(a_{k-1},a_k)$. Then, it is not hard to show that noncrossing partitions are exactly the partitions whose front representations are noncrossing, and nomincreasing partitions are exactly the partitions whose front representations are nonnesting.

\subsection{Noncrossing partition transform}
\label{subsec:noncrossing}

As stated in Corollary~\ref{cor:dyck_syt}: 
\begin{center}
\em The number of cm-labeled Dyck paths of semilength $n$ equals $\syt{n}$. 
\end{center}
This result is motivated by the noncrossing partition transform, which naturally relates its output to Dyck paths labeled by combinatorial objects enumerated by the input.

The noncrossing partition transform, as studied by Beissinger~\cite{Bei85} and Callan~\cite{Cal08}, may be defined in terms of partial Bell polynomials as follows:\footnote{The equivalence of this definition with \eqref{eq:d-Catalan} in the case $d=1$ is shown in \cite{BGW18}.} 

For a sequence $(x_n)$, define $(y_n)$ by
\begin{equation}\label{eq:BellFormula}
 y_0=1,\quad y_n = \sum_{k=1}^n \frac{1}{(n-k+1)!} B_{n,k}(1!x_1, 2!x_2, \dots) \ \text{ for } n\ge 1,
\end{equation}
where $B_{n,k}$ denotes the $(n,k)$-th partial Bell polynomial defined as
\begin{equation*}
  B_{n,k}(z_1,\dots,z_{n-k+1})=\sum_{\alpha\in\pi_k(n)} \tfrac{n!}{\alpha_1! \cdots \alpha_{n-k+1}!}\left(\tfrac{z_1}{1!}\right)^{\alpha_1}\cdots \left(\tfrac{z_{n-k+1}}{(n-k+1)!}\right)^{\alpha_{n-k+1}}
\end{equation*}
with $\pi_k(n)$ denoting the set of $(n-k+1)$-part partitions of $k$ such that $\alpha_1 + 2\alpha_2 +\cdots+(n-k+1)\alpha_{n-k+1}=n$.  As shown in \cite{BGMW16}, if $(x_n)$ is a sequence of nonnegative integers, $y_n$ gives the number of Dyck paths of semilength $n$ such that each $j$-ascent may be labeled in $x_j$ different ways. As expected, if $(x_n)$ is the sequence of ones, then $y_n$ gives the sequence of Catalan numbers. In general, $y_n$ enumerates configurations obtained by adorning the ascents with structures whose elements are counted by $(x_n)$. 

Let $a_j$ denote the number of all possible cm-labels for an ascent of length $2j$. This is the number of connected matchings on $[2j]$ and is given by the sequence \oeis{A000699}:
\[
1, 1, 4, 27, 248, 2830, 38232, 593859, \dots.
\] 
Therefore, if we define the sequence $(x_n)$ by 
\begin{gather*}
  x_1=1, \\
  x_{2n+1} = 0 \text{ and }  x_{2n} = a_n \text{ for } n\ge 1,
\end{gather*}
then from Corollary~\ref{cor:dyck_syt} and equation \eqref{eq:BellFormula} we deduce that
\begin{equation}\label{eq:sytBell}
  \syt{n} = \sum_{\ell=1}^n \frac{1}{(n-\ell+1)!}B_{n,\ell}(1!, 2!a_1,0,4!a_2,0, \dots).
\end{equation}

Observe that $\syt{n}$ is a special case of the sequence
\begin{equation}\label{eq:alphaFamily}
  y^{(\alpha)}_n = \sum_{\ell=1}^n \frac{1}{(n-\ell+1)!} B_{n,\ell}(1!\alpha, 2!a_1,0,4!a_2,0, \dots)
\end{equation}
that counts the number of cm-labeled Dyck paths of semilength $n$, where singletons (ascents of length 1) may be colored in $\alpha\in\mathbb{N}_0$ different ways. The case $\alpha=0$ means that no singletons are allowed. In this case, $y^{(0)}_{2n-1}=0$ for all $n\ge 1$ while $y^{(0)}_{2n}$ gives the number of perfect matchings on $[2n]$, which are counted by the double factorials $(2n-1)!!$.

Another interesting instance of \eqref{eq:alphaFamily} is when $\alpha=2$, i.e.\ each singleton may be colored in two ways. In this case, \eqref{eq:alphaFamily} gives the sequence \oeis{A005425} whose $n$th term gives the number of involutions on $[n]$ whose fixed points can each be colored in two different ways.  

\subsection{Generating functions}

Let $A(t)$ be the the generating function for the number of connected matchings on $[2n]$, and let $Y(t)$ be the  corresponding function that enumerates SYT with $n$ boxes. Equation \eqref{eq:sytBell} implies that $Y(t)$ is the noncrossing partition transform of $X(t) = t + A(t^2)$. Thus, in terms of generating functions, this means (cf.\ Callan \cite[\S4]{Cal08})
\[
  tY(t) = \left(\frac{t}{1+X(t)}\right)^{\langle-1\rangle},
\]
where $^{\langle-1\rangle}$ denotes compositional inverse.  In other words,
\begin{equation}\label{eq:Y=NXPtransformX} 
  Y(t) - 1 = X(tY(t)), \text{ or equivalently, }\, (1-t)Y(t) = 1 + A(t^2 Y(t)^2). 
\end{equation}
Further, if $P(t)$ is the generating function for the number of perfect matchings on $[2n]$, then $P(t^2)$ is the noncrossing partition transform of $A(t^2)$, and
\begin{equation*}
  1+ P(t^2) = \frac{1}{t}\left(\frac{t}{1+A(t^2)}\right)^{\langle-1\rangle}.
\end{equation*}
This implies
\begin{equation*}
  \left(t(1+ P(t^2))\right)^{\langle-1\rangle} = \frac{t}{1+A(t^2)} \;\text{ and }\;
  P\left(\frac{t^2}{(1+A(t^2))^2}\right) = A(t^2).
\end{equation*}
Combining this identity with \eqref{eq:Y=NXPtransformX}, we obtain
\begin{equation*}
  P\left(\frac{t^2Y(t)^2}{(1-t)^2Y(t)^2}\right) = A(t^2 Y(t)^2) = (1-t)Y(t) -1,
\end{equation*}
which implies
\begin{equation*}
  Y(t) = \frac{1+P(t^2/(1-t)^2)}{1-t}.
\end{equation*}

While this formula is known \oeis{A001006}, our approach using the noncrossing partition transform gives the same identity when restricted to $k$-noncrossing perfect matchings on $[2n]$ and SYT with $n$ boxes and height at most $2k-1$. In other words, if $P_k(t)$ denotes the generating function for the number of $k$-noncrossing perfect matchings on $[2n]$, and if $Y_k(t)$ enumerates SYT with $n$ boxes and height at most $2k-1$, then
\begin{equation*}
  Y_k(t) = \frac{1+P_k(t^2/(1-t)^2)}{1-t}.
\end{equation*}
This is the elegant expression we promised in the introduction. As stated in the survey \cite{Mis18+}, no explicit expression for the coefficients of $Y_k(t)$ for $k>3$ appears in the literature. For some values of $k$, these sequences are listed in \cite{OEIS} as follows:

\smallskip
\begin{center}
\begin{tabular}{r|c|c} \rule[-1.2ex]{0ex}{4ex}
  $k$ & $k$-noncrossing matchings & SYT of height $\le 2k-1$ \\ \hline
  \rule[0ex]{0ex}{3ex} 2 & \oeislink{A000108} & \oeislink{A001006} \\
  3 & \oeislink{A005700} & \oeislink{A049401} \\
  4 & \oeislink{A136092} & \oeislink{A007578} \\
  5 & \oeislink{A251598} & \oeislink{A212915}
\end{tabular}
\end{center}

Conjecturally, the number of SYT with $n$ boxes and height at most seven is given by
\begin{equation*}
\sum_{\ell=0}^{\lfloor{n/2}\rfloor}\binom{n}{2\ell}
\sum_{j=0}^{\ell+1} \frac{180(2\ell)!}{j!(j+4)!(\ell-j+1)!(\ell-j+3)!}(2\ell-3j+2)C_{j+1}\ .
\end{equation*}



\begin{thebibliography}{99}

\bibitem{AsMa08} A.~Asinowski and T.~Mansour, Dyck paths with colored ascents,  {\em European J. Combin.} \textbf{29} (2008), 1262--1279.

\bibitem{Bei85} J.~S.~Beissinger, The enumeration of irreducible combinatorial objects, {\em J. Combin. Theory Ser. A} \textbf{38} (1985), no. 2, 143--169.
%
\bibitem{BGMW16} D.~Birmajer, J.~Gil, P.~McNamara, and M.~Weiner, Enumeration of colored Dyck paths via partial Bell polynomials,  {\em Lattice Path Combinatorics and Applications}, G.~Andrews, C.~Krattenthaler, A.~Krinik (Eds.), Springer, 2019, 155--165.
%
\bibitem{BGW18} D.~Birmajer, J.~Gil, and M.~Weiner, A family of Bell transformations, {\em Discrete Math.} \textbf{342} (2019), no.~1, 38--54.
%
\bibitem{BCFMM16} S.~Burrill, J.~Courtiel, E.~Fusy, S.~Melczer, M.~Mishna, Tableau sequences, open diagrams, and Baxter families, {\em European J. Combin.} \textbf{58} (2016), 144--165.
%
\bibitem{Cal08} D.~Callan, Sets, lists and noncrossing partitions, {\em J. Integer Seq.} \textbf{11} (2008), no.~1, Article 08.1.3.
%
\bibitem{Cal09} D.~Callan, A combinatorial survey of identities for the double factorial, preprint \href{https://arxiv.org/abs/0906.1317}{arXiv:0906.1317}, 2009.
%
\bibitem{C+07} W.~Y.~C.~Chen, E.~Y.~P.~Deng, R.~R.~X.~Du, R.~P.~Stanley, and C.~H.~Yan, Crossings and nestings of matchings and partitions, {\em Trans. Amer. Math. Soc.} \textbf{359} (2007), no.~4, 1555--1575.
%
\bibitem{Eu10} S.-P.~Eu, Skew-standard tableaux with three rows, {\em Adv. in Appl. Math.} \textbf{45} (2010), 463--469.
%
\bibitem{E+13} S.-P. Eu, T.-S. Fu, J.T. Hou, and T.-W. Hsu, Standard Young tableaux and colored Motzkin paths, {\em J. Combin. Theory Ser. A} \textbf{120} (2013), no.~7, 1786--1803.
%
\bibitem{Fra78} J.~Fran\c{c}on, Histoires de fichiers, {\em RAIRO Informat. Thor.} \textbf{12} (1978), no.~1, 49--62.
%
\bibitem{FV79} J.~Fran\c{c}on and G. Viennot, Permutations selon leurs pics, creux, doubles mont\'ees et double descentes, nombres d'Euler et nombres de Genocchi, {\em Discrete Math.} \textbf{28} (1979), no.~1, 21--35.
%
\bibitem{GX18+} A.~M.~Garsia and G.~Xin, On the sweep map for Fuss Catalan rational Dyck paths, preprint, \href{https://arxiv.org/pdf/1807.07458.pdf}{arXiv:1807.07458}, 2018.
%
\bibitem{Gud10} H.~H.~Gudmundsson, Dyck paths, standard Young tableaux, and pattern avoiding permutations, {\em Pure Math. Appl. (PU.M.A.)}, \textbf{21} (2010), no.~2, 265--284.
%
\bibitem{Kim11} J.~S.~Kim, Front representations of set partitions, {\em SIAM J. Discrete Math.} \textbf{25} (2011), no.~1, 447--461.
%
\bibitem{Kratt06} C.~Krattenthaler, Growth diagrams, and increasing and decreasing chains in fillings of Ferrers shapes, {\em Adv. Appl. Math.} \textbf{37} (3) (2006), 404--431.
%
\bibitem{Mis18+} M.~J.~Mishna, On standard Young tableaux of bounded height, {\it Recent Trends in Algebraic Combinatorics}, A.~Barcelo, G.~Karaali, R.~Orellana (Eds.), Springer, 2019, 281-303.
%
\bibitem{OEIS} OEIS Foundation Inc. The On-Line Encyclopedia of Integer Sequences, 2016.
%
\bibitem{Pec14} O.~Pechenik, Cyclic sieving of increasing tableaux and small Schr\"oder paths, {\em J. Combin. Theory Ser. A} \textbf{125} (2014), 357--378.
%
\bibitem{Reg81} A.~Regev, Asymptotic values for degrees associated with strips of Young diagrams, {\em Adv. in Math.} \textbf{41} (1981), no.~2, 115--136.
%
\bibitem{Sag01} B.~E.~Sagan, \emph{The symmetric group; representations, combinatorial algorithms, and symmetric functions,} second edition, Graduate Texts in Mathematics, 203, Springer-Verlag, 2001. 
%
\bibitem{Sta96} R.~P.~Stanley, Polygon dissections and standard Young tableaux, {\em J. Combin. Theory Ser. A}, \textbf{76} (1996), no.~1, 175--177.
%
\bibitem{Sta99} R.~P.~Stanley, {\em Enumerative combinatorics, Vol.~2}, Cambridge University Press, Cambridge, 1999.
%
\bibitem{Wett17} M.~Wettstein, Trapezoidal diagrams, upward triangulations, and prime Catalan numbers, {\em Discrete Comput. Geom.} \textbf{58} (2017), no.~3, 505--525. 
\end{thebibliography}
\end{document}